\documentclass[11pt]{article}

\usepackage{amsmath, amsfonts, multirow,amssymb, amsthm,mathrsfs}
\usepackage{mathtools}
\usepackage{enumerate}
\usepackage[usenames,dvipsnames]{pstricks}
\usepackage{pst-plot,pstricks-add}
\usepackage{graphicx}
\usepackage{color}
\usepackage[margin=1in]{geometry}
\usepackage{lineno}
\usepackage{float}
\usepackage[utf8x]{inputenc}
\usepackage{algorithmicx,algorithm,algpseudocode,siunitx}
\usepackage{tikz}
\usepackage{epstopdf}

\newcommand{\C}{\mathbb{C}}    
\newcommand{\N}{\mathbb{N}}    
\newcommand{\R}{\mathbb{R}}    
\newcommand{\Z}{\mathbb{Z}}    
\newcommand{\imag}{\mathrm{i}} 

\newcommand{\dR}{\mathbb{R}^d}

\newcommand{\dZ}{\mathbb{Z}^d}
\newcommand{\dlp}[1]{l_{#1}(\mathbb{Z}^d)}
\newcommand{\td}{{\boldsymbol{\delta}}}  
\newcommand{\bp}{\begin{proof}}
\newcommand{\ep}{\hfill  \end{proof} }
\newcommand{\be}{ \begin{equation} }
\newcommand{\ee}{ \end{equation} }

\newcommand{\wh}{\widehat}
\renewcommand{\le}{\leqslant}
\renewcommand{\ge}{\geqslant}
\newcommand{\bs}{\backslash}
\newcommand{\ol}{\overline}

\renewcommand{\le}{\leqslant}
\renewcommand{\ge}{\geqslant}

\newcommand{\feta}{\boldsymbol{\eta}}

\newtheorem{theorem}{Theorem}

\newtheorem{lemma}{Lemma}

\theoremstyle{remark}

\begin{document}

\title{Regularization with Multilevel Non-stationary Tight Framelets  for Image Restoration}
\author{Yan-Ran Li \thanks{$^1$College of Computer Science and Software Engineering,$^2$Guangdong Key Laboratory of Intelligent Information Processing and Shenzhen Key Laboratory of Media Security, Shenzhen University, Shenzhen, 518060, P. R. China; $^3$SZU Branch, Shenzhen Institute of Artificial Intelligence and Robotics for Society, China. Email:\texttt{: lyran@szu.edu.cn}. Research supported in part by the Shenzhen R\&D Program (JCYJ20180305124325555). }
\and
Raymond H. F. Chan \thanks{Department of Mathematics, City University of Hong Kong, Tat Chee Avenue, Kowloon Tong, Hong Kong. Research supported in part by HKRGC Grants No. CUHK14301718, CityU Grant: 9380101, CRF Grant C1007-15G, AoE/M-05/12.
Email: \texttt{rchan.sci@cityu.edu.hk}.}
\and Lixin Shen\thanks{Department of Mathematics, Syracuse University, Syracuse, NY 13244, USA. Email: \texttt{lshen03@syr.edu}. The work of L. Shen was supported in part
by the National Science Foundation under grant DMS-1913039.}
\and Xiaosheng Zhuang\thanks{Department of Mathematics, City University of Hong Kong, Tat Chee Avenue, Kowloon Tong, Hong Kong. Email: \texttt{xzhuang7@cityu.edu.hk}. Research was supported in part by the Research Grants Council of Hong Kong (Project no. CityU 11301419) and City University of Hong Kong (Project no. 7005497)} }

\maketitle

\begin{abstract}
Variational regularization models are one of the popular and efficient approaches for image restoration. The regularization functional in the model carries prior knowledge about the image to be restored.  The prior knowledge, in particular for natural images, are the first-order (i.e. variance in luminance) and second-order (i.e. contrast and texture) information. In this paper, we propose a model for image restoration, using a multilevel non-stationary tight framelet system that can capture the image's first-order and second-order information. We develop an algorithm to solve the proposed model and the numerical experiments show that the model is effective and efficient as compared to other higher-order models.
\end{abstract}


\section{Introduction}

The restoration of a degraded image may be modeled as
\begin{equation}\label{eq:model}
z = K u + \epsilon, \quad \quad
\end{equation}
where $u$ denotes the unknown image to be recovered, $K$ a blurring
matrix, $z$ an observed blurred image, and $\epsilon$ the noise.  In
general, $K$ is a singular or near-singular matrix and hence the
problem of finding the solution $u$ from   model \eqref{eq:model}
is ill-posed. To overcome the difficulties caused by the ill-posedness, regularization techniques such as total-variation regularization  and multiscale regularization are often adopted,  see \cite{chambolle:ieeeip:98,chan:sjsc:03,figueiredo:ieeeip:03,geman-geman:PAMI:84,rudin:physicaD:92}
and the references therein.  The resulting regularized image models have the following generic form
\begin{equation}\label{eq:finalmodel}
\min_{u} \{\mathcal{F}(u) + \alpha \mathcal{G}(u)\}, \quad \alpha>0
\end{equation}
where $\alpha$ is the regularization parameter, $\mathcal{F}$ represents the data fidelity term and $\mathcal{G}$ the regularization term. The fidelity term measures the closeness of the estimate obtained from \eqref{eq:finalmodel} to the data $z$ while the regularization term is used to arrive at a sensible solution. Generally speaking, model~\eqref{eq:finalmodel} integrates knowledge about how data is generated in the fidelity term $\mathcal{F}$ with the regularization  functional $\mathcal{G}$ that carries prior knowledge about the image to be restored.

Our main focus of this paper is to choose a proper regularization $\mathcal{G}$ in \eqref{eq:finalmodel} for image restoration. Here, a proper regularization  means a regularization functional that encodes prior knowledge about the image to be restored. Prior knowledge about images, in particular for natural images, includes first-order (i.e. variance in luminance) and second-order (i.e. contrast and texture) information \cite{Johnson-Baker:JOSA:2004}. One commonly used regularization term that exploits the first-order information is the bounded variation semi-norm  \cite{rudin:physicaD:92}
\begin{equation}\label{eq:TV}
\mathcal{G}(u):=\int_\Omega |\nabla u|,
\end{equation}
where the image $u$ is defined on the bounded set $\Omega \subset \mathbb{R}^d$. The corresponding model \eqref{eq:finalmodel}, referred as the total variation (TV) based image restoration model, performs incredibly well especially if the image to be reconstructed is piecewise constant. The total variation functional does not penalize discontinuities in images and thus allows us
to recover the edges of the original image. However, it does not distinguish between jumps and smooth transitions, therefore it tends to give piecewise constant images with staircase artifacts. Due to this notably staircase phenomenon, the TV-based model is not suited for reconstructing images that are not nearly piecewise constant. It was pointed out in \cite{geman-reynolds:PAMI:92} that whereas the reconstruction generated with the first-order model will display jumps, the basic geometric structure of the original intensity surfaces is missing, even if it appears in the data.  It was further mentioned that using higher order models, these artifacts from the first-order model can be eliminated and some of the fine geometric structures, particularly planar and quadric patches, of the original image can be recovered.

One of the earliest models using higher derivatives was proposed in \cite{Chambolle-Lions:NM:97} where the infimal convolution of the first and second order derivatives was proposed as regularizer
\begin{equation}\label{eq:CL-regu}
\mathcal{G}(u):=\inf_{v} \int_\Omega |\nabla u -\nabla v| + \alpha |\nabla(\nabla v)|.
\end{equation}
It approximates locally the gradient of the function $u$ by $\nabla v$, that itself has a low total variation.
Different second-order functionals for staircase reduction have been considered in other papers, for example, see   \cite{Chan-Marquina-Mulet:00,Lysaker-Lundervold-Tai:IEEEIP:2003}. Based on tensor algebra, the regularizer with derivatives of arbitrary order was introduced in \cite{Bredies-Kunisch-Pock:SIAMIS:2010}. The corresponding regularizer was called total generalized variation (TGV). In particular, the TGV of second-order is
\begin{equation}\label{eq:TGV-regu}
\mathcal{G}(u):= \min_{v} \int_\Omega (\alpha_1 |\nabla u-v| + \alpha_2 |\mathcal{E}(v)|),
\end{equation}
where the parameters $\alpha_1, \alpha_2$ are positive, and
$$
\mathcal{E}(v)=\begin{bmatrix}
\partial_1 v_1 & \frac{1}{2}(\partial_1 v_2 + \partial_2 v_1) \\
\frac{1}{2}(\partial_1 v_2 + \partial_2 v_1) & \partial_2 v_2
\end{bmatrix}
$$
with $v_1$ and $v_2$ being the components of $v$. Note that for twice differentiable $u$, $\mathcal{E}(\nabla u)$ is the Hessian of $u$. We note that the TGV of second-order \eqref{eq:TGV-regu} is similar to, but structural different from, the regularizer \eqref{eq:CL-regu}. The use of TGV and its variants in a plethora of applications has been reported in \cite{Bredies-Holler:SIAMIS2015,Setzer-Steidl-Teuber:CMS:2011} and the references therein.

Motivated from the fact that an image/signal naturally has a hierarchical structure and allows to be represented in a multiscale structure, we exploit this structure to formulate a regularization term $\mathcal{G}$ in \eqref{eq:finalmodel} that is different from the aforementioned ones. To this end, we first construct a two-level non-stationary tight framelet system that is suitable for representing images to be restored. More specifically, the tight framelet system in the first level is the directional Haar framelet (DHF) system  introduced in our recent work \cite{Li-Chan-Shen-Hsu-Tseng:SIAMIS:16} while the one in the second level is constructed from the discrete cosine transform (DCT). We then use the framelet coefficients of an image under this two-level non-stationary tight framelet system to formulate the regularization term $\mathcal{G}$. More precisely, the framelet coefficients of the image with the DHF consist of the first-order information of the image in the vertical, horizontal, and $\pm 45^{\si{\degree}}$ directions. As a result, the regularization term $\mathcal{G}$ contains not only the TV term but also ameliorates it by including the diagonal information. The coarse approximation to this image resulting from the low-pass filter of the DHF, considered as a smooth version of this image, will facilitate the extraction of the second-order information of the image. As shown in our previous work \cite{Li-Shen-Suter:IEEEIP:13}, the second-order information of the image can be reliably  extracted from the DCT-based tight framelet coefficients of this smoothed image. Our proposed regularization term $\mathcal{G}$ also includes these second-order information. We remark that the success of tight framelets have been proven to be useful in image processing, see, e.g., \cite{Cai-Dong-Osher-Shen:JAMS:2012,Chan-Riemenschneider-Shen-Shen:ACHA:04,Li-Dai-Shen:IEEECSVT:10,Li-Shen-Dai-Suter:IEEEIP:11,Li-Shen-Suter:IEEEIP:13,shen:IEEEIP:06} and the references therein. However, despite that our two-level non-stationary tight framelet system is new, the proposed regularization is also different from the existing ones in the following perspectives:
\begin{itemize}
\item Due to the DHF, our regularization assimilates the advantages of both the total variation regularization and other framelet regularizations, and remedies their drawbacks. On the one hand, the filters associated with DHF have the shortest support among all tight framelet systems, therefore, it can suppress ringing artifacts arising from other framelet regularizations. In comparison, the filters associated with the 2-dimensional orthogonal Haar wavelet have the shortest support only among all compactly supported orthogonal wavelets.
    On the other hand, the diagonal first-order information provided by the DHF can reduce the staircase artifacts (or block effect) arising from the classical TV regularization. In comparison, the 2-dimensional orthogonal Haar wavelet only provide first-order information in the vertical and horizontal directions.

\item We exploit the second-order information of the underlying image from its smoothed version rather than from the image itself. The main idea behind it is that the high frequency spatial information of the image will be suppressed in its smoothed one and therefore the second-order information of the image will be faithfully computed, in particularly, for images with high degree of noise.

\item Finally, the properties of the tight framelet can be easily exploited to develop algorithms with computational efficiency and to analyze the convergence of the resulting algorithms.
\end{itemize}

To summarize, the proposed regularizer $\mathcal{G}$ contains the first and second order information of the image to be constructed  for \eqref{eq:finalmodel}. The resulting optimization problem \eqref{eq:finalmodel} can be efficiently solved and the efficiency and accuracy of this regularizer will be confirmed for image restoration.

The rest of this paper is organized as follows. In Section \ref{sec:model-algorithm} we first briefly review the tight framelet systems,  we then propose an image restoration model regularized by a two-level non-stationary tight framelet system and develop an algorithm to solve this model.  The performance of the proposed model for image restoration is presented in Section~\ref{sec:Experiments}.



\section{Model and Algorithm with Multi-Level Non-Stationary Tight Framelets}\label{sec:model-algorithm}
This section consists of three parts. In the first part, we briefly review the multi-level non-stationary tight framelet systems. In the second part, we propose our image restoration model using a two-level non-stationary tight framelet system. In the last part, we propose an algorithm to solve the resulting optimization problem.

\subsection{Multi-Level Non-Stationary Tight Framelets}\label{sec:FrameletSystems}
Tight framelets are closely related to filter banks.
A tight framelet filter bank can be used to (sparsely) represent data sequences through its associated discrete framelet transforms as well as its underlying discrete affine system \cite{Han:MMNP:13}. Before proceeding to their connections,  let us  recall some definitions and notation first.

By $l(\Z^d)$ we denote the set of all sequences and $\dlp{0}$ the set of all finitely supported sequences. A \emph{filter} or {\it mask} $h=\{h(k)\}_{k\in \dZ}: \dZ\rightarrow \C$ on $\dZ$ is a sequence in $\dlp{0}$. For a filter $h\in \dlp{0}$, its \emph{Fourier series} is defined to be $\wh{h}(\xi):=\sum_{k\in \dZ} h(k)e^{-\imag k\cdot\xi}$ for $\xi\in \dR$, which is a $2\pi\dZ$-periodic trigonometric polynomial. In particular, by $\td$ we denote \emph{the Dirac sequence} such that $\td(0)=1$ and $\td(k)=0$ for all $k\in\dZ\bs\{0\}$. Throughout the paper, we assume  the tight framelets are \emph{dyadic dilated}, that is, the \emph{dilation matrix} is $2I_d$ with $I_d$ the $d\times d $ identity matrix.

For filters $\tau_0,\tau_1,\ldots,\tau_s\in \dlp{0}$,
we say that a filter bank $\{\tau_0;\tau_1,\ldots,\tau_s\}$ is \emph{a ($d$-dimension dyadic) tight framelet filter bank} if
\be \label{tffb}
\sum_{\ell=0}^s \wh{\tau_\ell}(\xi)\ol{\wh{\tau_\ell}(\xi+\pi \omega)}=\td(\omega),\qquad \forall\, \xi\in \dR, \omega\in \{0,1\}^d,
\ee
where for a number $x\in\C$, $\bar{x}$ denotes its complex conjugate. Equation \eqref{tffb} is equivalent to the \emph{perfect reconstruction property} of the discrete framelet transforms associated with the filter bank $\{\tau_0;\tau_1,\ldots,\tau_s\}$ (\cite[Theorems 1.1.1 and 1.1.4]{Han:book}). The filter $\tau_0$ is usually a \emph{low-pass filter} satisfying $\wh{\tau_0}(0)=1$ while $\tau_\ell$'s are the \emph{high-pass filters} satisfying $\wh{\tau_\ell}(0)=0$ for $\ell \ge 1$.

In practice, multi-level decomposition and reconstruction of data using \emph{discrete framelet transform} associated with tight framelet filter banks are commonly used in order to exploit the sparse property of the data.   Moreover, in signal/image processing, translation invariance property of a discrete framelet transform is desirable especially in the scenario of signal denoising/inpainting. To preserve the translation invariance property, one usually considers the redundant version of discrete framelet transform, that is,  the \emph{undecimated discrete framelet transform} (UDFmT). More precisely, denote a filter bank at level $j$ as $\feta_j:=\{\tau_0^j; \tau_1^j,\ldots, \tau_{s_j}^j\}$ and consider a sequence $\{\feta_j : j=1,\ldots, J\} =\cup_{j=1}^J \{\tau_0^j; \tau_1^j,\ldots, \tau_{s_j}^j\}$ of $J$ filter banks with $j=J\ge 1$ being the finest level and $j=1$ being the coarsest level.   Let the \emph{convolution} operation * be defined by $
[h*v](\gamma):=\sum_{k\in\Z^d} h(\gamma-k)v(k)$, for  $v\in l(\Z^d), \, h\in l_0(\Z^d),\, \gamma\in\Z^d$, and the \emph{upsampling operator} $\uparrow m$  with $m\in\N$ be given by
\[
{[v\uparrow m]}(\gamma):=
\begin{cases}
v(m^{-1}\gamma), &\mbox{if $m^{-1}\gamma\in\Z^d$};\\
0, &\mbox{otherwise}.
\end{cases}
\]
For a filter $h$, let $h^\star$ be a filter defined by $h^\star(k)=\ol{h(-k)}$, $k\in\Z^d$. Then, for a given input data sequence $v=v_J$,  the UDFmT includes (i) Decomposition:
\begin{equation}\label{fmt:dec}
v_{j-1} = v_j*((\tau_0^j)^\star\uparrow 2^{J-j}),\quad w_{j-1;\ell} = v_j*((\tau_\ell^j)^\star\uparrow 2^{J-j}),\quad \ell = 1,\ldots,s_j,\quad j = J,\ldots, 1,
\end{equation}
and  (ii) Reconstruction:
\begin{equation}\label{fmt:rec}
v_j = v_{j-1}*(\tau_0^{j}\uparrow 2^{J-j})+\sum_{\ell=1}^sw_{j-1;\ell}*(\tau_\ell^j\uparrow 2^{J-j}),\quad  j = 1,\ldots, J.
\end{equation}
One can show that if each filter bank $\{\tau_0^j;\tau_1^j,\ldots, \tau^j_{s_j}\}$ satisfies  the \emph{partition of unity condition}:
$\sum_{\ell=0}^{s_j} |\wh{\tau_\ell^j}(\xi)|^2 = 1$, $\xi\in\R^d$,
then any input data sequence  $v\in l(\Z^d)$ can be perfectly reconstructed via \eqref{fmt:rec} from its framelet coefficient sequences $\{v_0\}\cup\{w_{j;\ell} : \ell=1,\ldots, s_j\}_{j=1}^J$ decomposed from \eqref{fmt:dec}. The framelet system associated with such a sequence $\{\feta_j : j=1,\ldots, J\}$ is then called a \emph{multi-level non-stationary tight framelet system}.

In this paper, we consider $J=2$, that is, two-level non-stationary tight framelet system.  One can of course consider $J>2$. However, in terms of efficiency and simplicity, $J=2$ is the best choice for the development of this paper.

\subsection{Regularization with a Two-level Non-stationary Tight Framelet System}
In this subsection, we integrate two different tight framelet systems as a two-level non-stationary tight framelet system which will be exploited for the optimization problem~\eqref{eq:finalmodel}.

The tight framelets in the first level is the directional Haar framelet (DHF) system proposed in \cite{Li-Chan-Shen-Hsu-Tseng:SIAMIS:16}. The filters associated with this DHF are
\begin{equation*}\label{eq:dir-Haar}
\begin{array}{llll}
\tau_0=\frac{1}{4}\begin{bmatrix}
          1 & 1  \\
          1 & 1
        \end{bmatrix},&
\tau_1=\frac{1}{4}\begin{bmatrix}
          1 & 0  \\
          0 & -1
        \end{bmatrix}, &
\tau_2=\frac{1}{4}\begin{bmatrix}
          0 & -1  \\
          1 & 0
        \end{bmatrix}, &
\tau_3=\frac{1}{4}\begin{bmatrix}
          1 & -1  \\
          0 & 0
        \end{bmatrix}, \\
        &
\tau_4=\frac{1}{4}\begin{bmatrix}
          1 & 0  \\
          -1 & 0
        \end{bmatrix}, &
\tau_5=\frac{1}{4}\begin{bmatrix}
          0 & 0  \\
          1 & -1
        \end{bmatrix}, &
\tau_6=\frac{1}{4}\begin{bmatrix}
          0 & 1  \\
          0 & -1
        \end{bmatrix}.
\end{array}
\end{equation*}
As two-dimensional filters, the indices of the entries (top-left, top-right, bottom-left, and bottom-right) in each filter are $(0,0)$, $(0,1)$, $(1,0)$, and $(1,1)$, respectively. The first filter $\tau_0$  is a low-pass filter and the rest are high-pass filters that have the ability to provide directional information of an image when these filters are applied to the image. More precisely, the filters $\tau_1$ and $\tau_2$ act as the first-order difference operators in the $45^{\si{\degree}}$ and $135^{\si{\degree}}$ directions, respectively. The results of these two filters convolving with an image will highlight changes in intensity of the image in these two diagonal directions. The filters $\tau_3$ and $\tau_5$ are the first-order difference operators in the horizontal direction while the filters $\tau_4$ and $\tau_6$ are the first-order difference operators in the vertical direction. The convolutions of these filters with the underlying image are the coefficients of the image under the corresponding filters, which are the multiplications of some associated transformation matrices with the image.

Now, we propose a generic regularization term based on DHF. Let $u \in \mathbb{R}^n$ be the vector representing the column-stacked version of an image. We denote by $M_\kappa$ the associated matrix representation of the filters $\tau_\kappa$, $\kappa=0,1,\ldots, 6$, under a proper boundary condition. We further denote
\begin{equation}\label{def:Harr-low-high-Matrices}
B_{1\ell}:=M_0 \quad \mbox{and} \quad B_{1h}:=[M_1^\top, \ldots,  M_6^\top]^\top.
\end{equation}
By the tight frame property of $\{\tau_0;\tau_1, \ldots, \tau_s\}$, these two matrices satisfy the following perfect reconstruction condition
$$
B^\top_{1\ell}B_{1\ell} + B^\top_{1h}B_{1h}=I.
$$
Let $\Phi_{1\Lambda}: \mathbb{R}^{6n} \to \mathbb{R}$ be defined through a function $\varphi_1: \mathbb{R}^{6} \to \mathbb{R}$ and a non-negative parameter vector $\Lambda=\begin{bmatrix}\lambda_1, \lambda_2, \ldots, \lambda_n \end{bmatrix}$ as follows
\begin{equation}\label{def:Phi1}
\Phi_{1\Lambda}(v):=\sum_{i=1}^n \lambda_i \varphi_1(v_i,v_{i+n},\ldots, v_{i+5n}).
\end{equation}
With this function $\Phi_{1\Lambda}$, we propose a functional based on DHF in the following form
\begin{equation}\label{eq:G-1}
\mathcal{G}_1(u):=\Phi_{1\Lambda}(B_{1h}u),
\end{equation}
from which the TV regularization and its variants can be derived by properly chosen $\varphi_1$ in \eqref{def:Phi1}. For example, if we choose
$\varphi_1(x_1,x_2,x_3,x_4,x_5,x_6)=|x_3|+|x_4|$, the regularization in \eqref{eq:G-1} is reduced to the so-called anisotropic TV; If we choose
$\varphi_1(x_1,x_2,x_3,x_4,x_5,x_6)=\sqrt{|x_3|^2+|x_4|^2}$,  the regularization in \eqref{eq:G-1} is reduced to the so-called isotropic TV.

We choose, in this paper,
\begin{equation}\label{eq:varphi-1-DHF}
\varphi_1(x_1,x_2,x_3,x_4,x_5,x_6)=\sqrt{|x_1|^2+|x_2|^2}+\sqrt{|x_3|^2+|x_4|^2}.
\end{equation}
One of the advantages of the regularization $\mathcal{G}_1$ with $\varphi_1$ given in \eqref{eq:varphi-1-DHF} is that it assimilates the advantages of both total variation and wavelet regularizations and remedies their drawbacks. The way of avoiding or suppressing ringing artifacts arising from wavelet regularizations is to choose a wavelet system whose filters have small supports. The filters associated with the 2-dimensional orthogonal Haar wavelet have the shortest support among all compactly supported orthogonal wavelets, but the staircase artifacts (or blocky effect) will appear in the neighborhoods of edges in the directions about $\pm 45^{\si{\degree}}$. Since $\varphi_1$ in \eqref{eq:varphi-1-DHF} includes the diagonal first-order information from the filters $\tau_1$ and $\tau_2$, the staircase artifacts can be reduced.

The tight framelet in the second level is generated from the standard $3 \times 3$ DCT-II orthogonal matrix whose three rows are $c_0=\frac{\sqrt{3}}{3}[1, 1, 1]$, $c_1=\frac{\sqrt{2}}{2}[1, 0, -1]$, and $c_2=\frac{\sqrt{6}}{6}[1, -2, 1]$. In the sequel, this system is referred to as the DCT-based tight framelet system. The filters of the DCT-based tight framelet system are
$\tau_{3i+j}=\frac{1}{3}c_i^\top c_j$  with $i,j \in \{0,1,2\}$, where $\tau_{0}$ is the low-pass filter and the others are high-pass filters. Here, for simplicity of notation, we use  $\tau_\kappa$ to denote the filters associated with both the DHF or DCT-based tight framelet. The expansions of these filters are
\begin{equation*}\label{eq:dir-DCT}
\begin{array}{rrr}
\tau_0=\frac{1}{9}\begin{bmatrix}
          1 & 1 & 1\\
          1 & 1 & 1\\
          1 & 1 & 1
        \end{bmatrix},&
\tau_1=\frac{\sqrt{6}}{18}\begin{bmatrix}
          1 & 0 & -1\\
          1 & 0 & -1\\
          1 & 0 & -1
        \end{bmatrix},&
\tau_2=\frac{\sqrt{2}}{18}\begin{bmatrix}
          1 & -2 & 1\\
          1 & -2 & 1\\
          1 & -2 & 1
        \end{bmatrix}, \\

\tau_3=\frac{\sqrt{6}}{18}\begin{bmatrix}
          1 & 1 & 1\\
          0 & 0 & 0\\
          -1&-1 & -1
        \end{bmatrix},&
\tau_4=\frac{1}{6}\begin{bmatrix}
          1 & 0 & -1\\
          0 & 0 & 0\\
          -1& 0 & 1
        \end{bmatrix},&
\tau_5=\frac{\sqrt{3}}{18}\begin{bmatrix}
          1 & -2 & 1\\
          0 &  0 & 0\\
          -1& 2 & -1
        \end{bmatrix}, \\

\tau_6=\frac{\sqrt{2}}{18}\begin{bmatrix}
          1 & 1 & 1\\
          -2 & -2 & -2\\
          1 & 1 & 1
        \end{bmatrix}, &
\tau_7=\frac{\sqrt{3}}{18}\begin{bmatrix}
          1 & 0 & -1\\
          -2 &  0 & 2\\
          1& 0 & -1
        \end{bmatrix}, &
\tau_8=\frac{1}{18}\begin{bmatrix}
          1 & -2 & 1\\
          -2 &  4 & -2\\
          1& -2 & 1
        \end{bmatrix}.
\end{array}
\end{equation*}
The filters $\tau_1$ and $\tau_3$, known as the Prewitt operator in image processing, are used to compute an approximation of the gradient (i.e., the first-order information) of the image intensity function. The convolution of $\tau_1$ (resp. $\tau_3$) with an image gives the horizontal (resp. vertical) changes of the image intensity and they compute changes of intensity with smoothing due to $\tau_1=\frac{1}{3}c_0^\top c_1$ and $\tau_3=\frac{1}{3}c_1^\top c_0$.  The filter $\tau_2$ (resp. $\tau_6$) computes the discrete second-order difference in vertical (resp. horizontal) direction  with smoothing due to $\tau_2=\frac{1}{3}c_0^\top c_2$ and $\tau_6=\frac{1}{3}c_2^\top c_0$.  The other filters $\tau_4$, $\tau_5$, $\tau_7$, and $\tau_8$ perform like discrete high-order difference operators.

We should note that the first-order derivative operators exaggerate the effects of noise while the second-order derivatives will exaggerated noise twice as much \cite{gonzalez:93}. Therefore, the applicability of the second-order derivatives is limited to images with low noise level. Motivated from the Laplacian of a Gaussian (LOG) and difference of Gaussian (DOG) operators in computer vision,  see, for example, \cite{Lowe:IJCV:2004,Marr-Hildreth:PRSL:1980}, we propose to take the second-order derivatives on the blurred or smoothed images in order to reduce the effect of the presence of noise in an image. To this end, we denote by $P_\kappa$ the matrix representation of the filters $\tau_\kappa$, $\kappa=0,1,\ldots 8$, under a proper boundary condition. Let us define
\begin{equation}\label{def:DCT-low-high-Matrices}
B_{2\ell}:=P_0 \quad \mbox{and} \quad B_{2h}:=[P_1^\top, \ldots,  P_8^\top]^\top.
\end{equation}
We have that
$$
B^\top_{2\ell}B_{2\ell} + B^\top_{2h}B_{2h}=I.
$$
Let $\Phi_{2\Theta}: \mathbb{R}^{8n} \to \mathbb{R}$ be defined through a nonnegative parameter sequence $\Theta=\{\theta_i=(\theta_{i1},\theta_{i2}, \ldots, \theta_{i8}) \in \mathbb{R}^8: 1\le i \le n\}$ with non-negative elements as follows
\begin{equation}\label{def:Phi2}
\Phi_{2\Theta}(v):=\sum_{i=1}^n  \|[\theta_{i1}v_i, \theta_{i2}v_{i+n}, \ldots, \theta_{i8}v_{i+7n}]\|_1,
\end{equation}
where $\|\cdot\|_1$ denotes the $\ell_1$ norm. With this function $\Phi_{2\Theta}$, we propose a functional based on the DCT-based tight framelet system  in the following form
\begin{equation}\label{eq:reg-DCT9}
\mathcal{G}_2(u):=\Phi_{2\Theta}(B_{2h}B_{1\ell}u),
\end{equation}
where $B_{1\ell}u$ is viewed as the smooth version of $u$.

All together, our proposed image restoration model is
\begin{equation}\label{def:OurRegu}
\min_{u} \{\mathcal{F}(u) + \mathcal{G}_1(u)+\mathcal{G}_2(u)\}.
\end{equation}
The efficiency of the regularization functional $\mathcal{G}_1(u)+\mathcal{G}_2(u)$ in \eqref{def:OurRegu} will be presented in Section~\ref{sec:Experiments} when it is compared with several possible regularization functionals formulated from the DHF and DCT-based tight framelet, and with other existing higher-order regularization functionals.

\subsection{Algorithm}

In this subsection, we specify the data fidelity $\mathcal{F}$  in \eqref{def:OurRegu}. For Gaussian noise, the natural choice for $\mathcal{F}$ is $\mathcal{F}(u)=\frac{1}{2}\|Ku-z\|^2$ where $\| \cdot \|$ denotes either the vector 2-norm or matrix 2-norm. That is, the optimization problem we consider here is
\begin{equation}\label{eq:our-opt}
\min_{u \in [0, 1]^n} \frac{1}{2}\|Ku-z\|^2 + \Phi_{1\Lambda}(B_{1h}u)+\Phi_{2\Theta}(B_{2h}B_{1\ell}u),
\end{equation}
where $\Phi_{1\Lambda}$ is given in \eqref{def:Phi1} and $\Phi_{2\Theta}$ is given in \eqref{def:Phi2}. Here, we assume that all pixel values of an image are in $[0,1]$.

We next introduce our notation and recall some necessary background from convex analysis.  The class of all lower semicontinuous convex functions  $f: \mathbb{R}^d \rightarrow (-\infty, +\infty]$ such that $\mathrm{dom} \; f:=\{x \in \mathbb{R}^d: f(x)<+\infty\} \neq \emptyset$ is denoted by $\Gamma_0(\mathbb{R}^d)$. The indicator function of a closed convex set $C$ in  $\mathbb{R}^d$  is defined, at $u \in \mathbb{R}^d$, as
$$
\iota_C(u): =\left\{
               \begin{array}{ll}
                 0, & \hbox{if $u\in C$,} \\
                 +\infty, & \hbox{otherwise.}
               \end{array}
             \right.
$$
Clearly, the indicator function $\iota_C$ is in $\Gamma_0(\mathbb{R}^d)$ for any closed nonempty convex set $C$.

For a function $f \in \Gamma_0(\mathbb{R}^d)$, the proximity operator of $f$ with parameter $\lambda$, denoted by $\mathrm{prox}_{\lambda f}$, is a mapping from $\mathbb{R}^d$ to itself, defined for a given point $x \in \mathbb{R}^d$ by
$$
\mathrm{prox}_{\lambda f} (x):=\mathop{\mathrm{argmin}} \left\{\frac{1}{2} \|u-x\|^2 + \lambda f(u): u \in \mathbb{R}^d \right\}.
$$
We also need the notation of conjugate. The conjugate of $f \in \Gamma_0(\mathbb{R}^d)$ is the function $f^* \in \Gamma_0(\mathbb{R}^d)$ defined at  $x \in \mathbb{R}^d$ by $f^*(x):=\sup \{\langle u, x\rangle -f(u): u \in \mathbb{R}^d\}$. A key property of the proximity operators of $f$ and its conjugate is
\begin{equation}\label{prox-f-f*}
\mathrm{prox}_{\lambda f} (x)+\lambda \mathrm{prox}_{\lambda^{-1} f^*} (x/\lambda)=x,
\end{equation}
which holds for all $x \in \mathbb{R}^n$ and any $\lambda>0$.

Now, we turn to the optimization problem~\eqref{eq:our-opt}.  Define
\begin{equation}\label{identify3function:1}
f(u)=\frac{1}{2}\|Ku-z\|^2, \quad g(u)=\iota_{[0,1]^n}, \quad p(s)=\Phi_{1\Lambda}(s_1)+\Phi_{2\Theta}(s_2), \quad \mbox{and} \quad
A=\begin{bmatrix}B_{1h} \\ B_{2h}B_{1\ell}\end{bmatrix},
\end{equation}
where $u\in \mathbb{R}^n$ and $s=(s_1,s_2)$ with $s_1 \in \mathbb{R}^{6n}$ and $s_2 \in \mathbb{R}^{8n}$. Then,  our optimization problem~\eqref{eq:our-opt} can be viewed as a special case of the optimization problem whose objective function is the sum of three lower semicontinuous convex functions in the form of
\begin{equation}\label{model:three-terms-general}
\min_{u \in \mathbb{R}^n} f(u)+g(u)+p(Au),
\end{equation}
where $A$ is a $d\times n$ matrix,  $f\in\Gamma_0(\mathbb{R}^n)$ is differentiable, $g \in \Gamma_0(\mathbb{R}^n)$, and $p\in \Gamma_0(\mathbb{R}^d)$.

Several algorithms have been developed for the optimization problem~\eqref{model:three-terms-general}, see, for example, \cite{Combettes-Pesquet:SVVA:2012,Condat:JOTA:2013,Li-Zhang:ACHA:2016,Yan:JSC:2018}. We adopt the algorithm given in \cite{Yan:JSC:2018} for problem~\eqref{model:three-terms-general} since it converges under a much weaker condition and can choose a larger step-size, yielding a faster convergence.  This algorithm, named as Primal-Dual Three-Operator splitting (PD3O), has the following iteration:
\begin{subequations}
    \begin{align}
    u^{k}&=\mathrm{prox}_{\gamma g}(v^k) \label{eq:Yan1}\\
    s^{k+1}&=\mathrm{prox}_{\delta p^*}\left((I-\gamma\delta AA^\top)s^k+\delta A(2u^k-v^k-\gamma \nabla f(u^k))\right) \label{eq:Yan2}\\
    v^{k+1}&=u^k-\gamma \nabla f(u^k)-\gamma A^\top s^{k+1} \label{eq:Yan3}
    \end{align}
\end{subequations}
One PD3O iteration can be viewed as an operator $\mathrm{T}_{\mathrm{PD3O}}$ such  that $(v^{k+1}, s^{k+1})=\mathrm{T}_{\mathrm{PD3O}}(v^{k}, s^{k})$.
The convergence analysis of PD3O is given in the following lemma.
\begin{lemma}[Sublinear convergence rate \cite{Yan:JSC:2018}] \label{lemma:Yan}
Let $f\in \Gamma_0(\mathbb{R}^n)$ and its gradient be Lipschitz continuous with constant $L$, let $g \in \Gamma_0(\mathbb{R}^n)$, and $p\in \Gamma_0(\mathbb{R}^d)$. Choose $\gamma$ and $\delta$ such that $\gamma < 2/L$ and $M=\frac{\gamma}{\delta}(I-\gamma\delta AA^\top)$ is positive definite. Let $(v^*, s^*)$ be any fixed point of $\mathrm{T}_{\mathrm{PD3O}}$, and $\{(v^k, s^k)\}_{k \ge 0}$ be the sequence generated by PD3O. Define $\|(v,s)\|_M:=\sqrt{\|v\|^2+\langle s, Ms \rangle}$. Then, the following statements hold.
\begin{itemize}
\item[(i)] The sequence $\{(\|(v^k,s^k)-(v^*,s^*)\|_M)\}_{k \ge 0}$ is monotonically nonincreasing.
\item[(ii)] The sequence $\{(\|(v^{k+1},s^{k+1})-(v^k,s^k)\|_M)\}_{k \ge 0}$ is monotonically nonincreasing. Moreover, $\|(v^{k+1},s^{k+1})-(v^k,s^k)\|_M^2=o\left(\frac{1}{k+1}\right)$.
\end{itemize}
\end{lemma}


To adapt PD3O for our optimization problem~\eqref{eq:our-opt} with $f$, $g$ and $p$, and the matrix $A$ given in \eqref{identify3function:1}, some preparations are provided in the following lemmas.
\begin{lemma}\label{lemma:PHI1-prox}
Let $\delta >0$ and $\Phi_{1\Lambda}$ be given in \eqref{def:Phi1}. For any $v\in \mathbb{R}^{6n}$,  if $y=  \mathrm{prox}_{\delta^{-1} \Phi_{1\Lambda}}(v)$, then
\begin{equation}\label{eq:yi-phi1}
y_{(i)}=\mathrm{prox}_{\delta^{-1} \lambda_i \varphi_{1}}(v_{(i)}),
\end{equation}
where $y_{(i)}=\begin{bmatrix}y_{i}&y_{i+n}&\cdots&y_{i+5n}\end{bmatrix}^\top$ and $v_{(i)}=\begin{bmatrix}v_{i}&v_{i+n}&\cdots&v_{i+5n}\end{bmatrix}^\top$. Furthermore, let $y_{(ij)}$ and $v_{(ij)}$ be $y_{i+(j-1)n}$ and  $v_{i+(j-1)n}$, respectively, for $i=1,\ldots , n$ and $j=1,\ldots, 6$, then
$$
\begin{bmatrix}y_{(i1)} \\ y_{(i2)} \end{bmatrix}=\left(1-\frac{\lambda_i\delta^{-1}}{\max\{\|\begin{bmatrix}v_{(i1)} & v_{(i2)} \end{bmatrix}\|, \lambda_i\delta^{-1}\}}\right)\begin{bmatrix}v_{(i1)} \\ v_{(i2)} \end{bmatrix},
$$
where the pair $(y_{(i3)}, y_{(i4)})$ is obtained by simply replacing $(v_{(i1)}, v_{(i2)})$  in the right hand side of the above formula by  $(v_{(i3)}, v_{(i4)})$, and $y_{(i5)}=v_{(i5)}$, $y_{(i6)}=v_{(i6)}$.
\end{lemma}
\begin{proof}
\ \ The proof is based on the block separable property of $\Phi_{1\Lambda}$ in \eqref{def:Phi1}. By the definition of proximity operator and equations \eqref{def:Phi1} and \eqref{eq:varphi-1-DHF},
\begin{eqnarray*}
\mathrm{prox}_{\delta^{-1} \Phi_{1\Lambda}}(v)&=&\mathop{\mathrm{argmin}} \left\{\frac{1}{2} \|u-v\|^2 + \delta^{-1} \Phi_{1\Lambda}(u): u \in \mathbb{R}^{6n} \right\}\\
&=&\mathop{\mathrm{argmin}} \left\{\sum_{i=1}^n \frac{1}{2} \|u_{(i)}-v_{(i)}\|^2 + \delta^{-1}\lambda_i\varphi_1(u_{(i)}): u_{(i)} \in \mathbb{R}^6, i=1,\ldots,n \right\}.
\end{eqnarray*}
Hence, equation~\eqref{eq:yi-phi1} holds.  Notice that $\varphi_1$ is also a block separable function. By using the definition of proximity again and the proximity operator of the $\ell_2$ norm (see, for example, \cite{Combettes-Wajs:MMS:05,Micchelli-Shen-Xu:IP-11}), we obtain the explicit expression for $y_{(i)}$ as given above.
\end{proof}

\begin{lemma}\label{lemma:PHI2-prox}
Let $\delta >0$ and $\Phi_{2\Theta}$ be given in \eqref{def:Phi2}. For any $v\in \mathbb{R}^{8n}$,  if $y=  \mathrm{prox}_{\delta^{-1} \Phi_{2\Theta}}(v)$, then
\begin{equation}\label{eq:yi-phi2}
y_{(i)}=\mathrm{prox}_{\delta^{-1} \|\cdot\|_1\circ \mathrm{diag}(\theta_i)} (\mathrm{diag}(v_{(i)}).
\end{equation}
where $y_{(i)}=\begin{bmatrix}y_{i}&y_{i+n}&\cdots&y_{i+7n}\end{bmatrix}^\top$ and $v_{(i)}=\begin{bmatrix}v_{i}&v_{i+n}&\cdots&v_{i+7n}\end{bmatrix}^\top$. Furthermore, let $y_{(ij)}$ and $v_{(ij)}$ be $y_{i+(j-1)n}$ and  $v_{i+(j-1)n}$, respectively, for $i=1,\ldots n$ and $j=1,\ldots, 8$, then
$$
y_{(ij)} = \max\{|v_{(ij)}|-\delta^{-1}\theta_{ij}, 0\} \mathrm{sgn}(v_{(ij)}).
$$
\end{lemma}
\begin{proof} \ \ The proof is based on the block separable property of $\Phi_{2\Theta}$ in \eqref{def:Phi2}. By the definition of proximity operator,
\begin{eqnarray*}
\mathrm{prox}_{\delta^{-1} \Phi_{2\Theta}}(v)&=&\mathop{\mathrm{argmin}} \left\{\frac{1}{2} \|u-v\|^2_2 + \delta^{-1} \Phi_{2\Theta}(u): u \in \mathbb{R}^{8n} \right\}\\
&=&\mathop{\mathrm{argmin}} \left\{\sum_{i=1}^n \frac{1}{2} \|u_{(i)}-v_{(i)}\|^2_2 + \delta^{-1}\|\mathrm{diag}(\theta_i)u_{(i)}\|_1: u_{(i)} \in \mathbb{R}^8, i=1,\ldots,n \right\}.
\end{eqnarray*}
Hence, equation~\eqref{eq:yi-phi2} holds. Furthermore, notice that $\mathrm{prox}_{\delta^{-1} \|\cdot\|_1\circ \mathrm{diag}(\theta_i)}$ is the well-known soft thresholding operator, the rest of result holds.
\end{proof}

\begin{lemma}\label{lemma:Phi1+Phi2}
Let $\Phi_{1\Lambda}$ be given in \eqref{def:Phi1} and $\Phi_{2\Theta}$ be given in \eqref{def:Phi2}. For any $v\in \mathbb{R}^{14n}$, write $v=(v_1,v_2)$ with $v_1 \in \mathbb{R}^{6n}$ and $v_2 \in \mathbb{R}^{8n}$, and  define $p(v)=\Phi_{1\Lambda}(v_1)+\Phi_{2\Theta}(v_2)$. Then, for any $\delta>0$,
$$
\mathrm{prox}_{\delta^{-1} p}(v)=\mathrm{prox}_{\delta^{-1} \Phi_{1\Lambda}}(v_1) \times \mathrm{prox}_{\delta^{-1} \Phi_{2\Theta}}(v_2).
$$
\end{lemma}
\noindent
The result in the above lemma comes from the block separability of the function $p$. Therefore, we omit its proof here.

To apply Lemma~\ref{lemma:Yan} to  problem~\eqref{eq:our-opt}, we verify all the requirements listed in Lemma~\ref{lemma:Yan}. First, for the function $f$ in \eqref{identify3function:1}, we have that $\nabla f(u) = K^\top (Ku-z)$, the gradient of $f$ is $\|K\|^2$-Lipschitz continuous. Next, we discuss the positive definiteness of the matrix $I-\gamma \delta AA^\top$.
\begin{lemma}\label{lemma:A-norm}
Let $A$ be given in \eqref{identify3function:1}. Then, for positive numbers $\gamma$ and $\delta$, the matrix $I-\gamma \delta AA^\top$ is positive semidefinite (or definite) if and only if $\gamma\delta \le 1$ (or $\gamma\delta < 1$).
\end{lemma}
\begin{proof}\ \ First, we show that  $\|A\|=1$. For any $u \in \mathbb{R}^n$, we have that
$$
u^\top A^\top A u = u^\top B_{1h}^\top B_{1h} u + u^\top B_{1\ell}^\top B_{2h}^\top B_{2h}B_{1\ell} u.
$$
Since $B_{2h}^\top B_{2h}+B_{2\ell}^\top B_{2\ell}=I$ and $B_{1h}^\top B_{1h}+B_{1\ell}^\top B_{1\ell}=I$, from the above we have that
$$
u^\top A^\top A u \le u^\top B_{1h}^\top B_{1h} u + u^\top B_{1\ell}^\top B_{1\ell} u = u^\top u.
$$
Hence $\|A\| \le 1$. Further, since the null space of $B_{1h}$ is non-empty, therefore, $\|A\|=1.$

Next, since $AA^\top$ is positive semi-definite and its largest eigenvalue is $1$, hence,  $I-\gamma \delta AA^\top$ is positive semidefinite (or definite) if and only if $\gamma\delta \le 1$ (or $\gamma\delta < 1$).
\end{proof}

The explicit form of $\mathrm{prox}_{\delta^{-1} p}$ is given in Lemma~\ref{lemma:Phi1+Phi2} with the help of Lemmas~\ref{lemma:PHI1-prox} and \ref{lemma:PHI2-prox}. Therefore, the proximity operator $\mathrm{prox}_{\delta p^*}$ can be computed via \eqref{prox-f-f*}. With the above preparation, the complete procedure for solving \eqref{eq:our-opt} based on \eqref{eq:Yan1}-\eqref{eq:Yan3} is described in Algorithm~\ref{Alg:Practical-1}. This algorithm is refereed to as TNTF (two-level non-stationary tight framelet) algorithm.
\begin{algorithm}[H]
\caption{Two-level Non-stationary Tight  Framelet (TNTF) Algorithm}
\label{Alg:Practical-1}
\begin{algorithmic}[1]
\State Set parameters    $\gamma<\frac{2}{\|K\|^2}$, $\gamma\delta<1$; pre-given parameters $\Lambda$ and $\Theta$.
\State Initialize  ${v}^0=0$ and $s^0=0$
\State Auxiliary variable $x^k$ and write $s^k=(s^k_1, s^k_2)$
\For{$k=1,2,\ldots$}
\State
\begin{subequations}
    \begin{align}
    u^{k}&=\mathrm{Proj}_{[0,1]}(v^k) \\
    x^{k}&=\gamma(B_{1h}^\top s^k_1 + B_{1\ell}^\top B_{2h}^\top s^k_2) - (2u^k-v^k)+\gamma K^\top (K u^k-z) \label{eq:Alg1-1}\\
    s^{k+1}_1&=(s^k_1-\delta B_{1h} x^k)-\delta \cdot \mathrm{prox}_{\delta^{-1}\Phi_{1\Lambda}}(\delta^{-1}(s^k_1-\delta B_{1h} x^k)) \label{eq:Alg1-2} \\
    s^{k+2}_1&=(s^k_2-\delta B_{2h} B_{1\ell} x^k)-\delta \cdot \mathrm{prox}_{\delta^{-1}\Phi_{2\Theta}}(\delta^{-1}(s^k_2-\delta B_{2h} B_{1\ell} x^k)) \label{eq:Alg1-3}\\
    v^{k+1}&=u^k-\gamma K^\top(Ku^k-z)-\gamma(B_{1h}^\top s^{k+1}_1 + B_{1\ell}^\top B_{2h}^\top s^{k+2}_2)\label{eq:Alg1-4}
    \end{align}
\end{subequations}
\EndFor
\end{algorithmic}
\end{algorithm}

The convergence analysis for Algorithm~\ref{Alg:Practical-1} is as follows.
\begin{theorem}\label{convergence-1}
Let $(v^*, s^*)$ be any fixed point of $\mathrm{T}_{PD3O}$ with $f$, $g$, $p$ and $A$  given in \eqref{identify3function:1}. Let $\{(v^k, s^k)\}_{k \ge 0}$ be the sequence generated by Algorithm~\ref{Alg:Practical-1}, where $s^{k}=(s^{k}_1, s^{k}_2)$. Choose $\gamma$ and $\delta$ such that $\gamma < 2/\|K\|^2$ and $\gamma \delta<1$. Define $M=\frac{\gamma}{\delta}(I-\gamma\delta AA^\top)$. Then,
the following statements hold.
\begin{itemize}
\item[(i)] The sequence $\{(\|(v^k,s^k)-(v^*,s^*)\|_M)\}_{k \ge 0}$ is monotonically nonincreasing.
\item[(ii)] The sequence $\{(\|(v^{k+1},s^{k+1})-(v^k,s^k)\|_M)\}_{k \ge 0}$ is monotonically nonincreasing. Moreover, $\|(v^{k+1},s^{k+1})-(v^k,s^k)\|_M^2=o\left(\frac{1}{k+1}\right)$.
\end{itemize}
\end{theorem}
\begin{proof}\ \
We know that the gradient of $f$ in \eqref{identify3function:1} is $\|K\|^2$-Lipschitz continuous. By Lemma~\ref{lemma:A-norm}, the matrix $M$ is positive definite if $\gamma \delta<1$, the result of this theorem follows immediately from Lemma~\ref{lemma:Yan}.
\end{proof}

Remark: the computational cost of Algorithm 1 depends on mainly two factors: the UDFmT used in steps (\ref{eq:Alg1-1}-e) and the total number of iterations for $k$.
The UDFmT decompositions include $B_{1h}x^k$, $B_{2h}B_{1\ell} x^k$ in (\ref{eq:Alg1-2}-d) while the UDFmT reconstructions include $B_{1h}^\top s_1^k$, $B_{1\ell}^T B_{2h}^\top s_2^k$ in \eqref{eq:Alg1-1} and $B_{1h}^\top s_1^{k+1}$ and $B_{1\ell}^\top B_{2h}^\top s_2^{k+2}$ in \eqref{eq:Alg1-4}.
Since UDFmT in Algorithm 1 uses convolutions with 7 DHF filters in the first level and 9 DCT-based filters in the second level,  the UDFmT can be implemented with computational cost $O(n)$, where $n$ is the number of pixels in $u$.
For the total number of iterations $k$ in Algorithm 1, it depends on when the algorithm converges and the maximum number $\mathbb{K}$ of iterations set manually. Consequently, the total computational cost for Algorithm 1 is $O(\mathbb{K}n)$. 

Finally, we discuss how to choose the parameters in the algorithm. In our tests below we choose $\gamma = 1.99, \delta= 0.5$ to ensure $\gamma\delta<1$.
Recall from (\ref{eq:varphi-1-DHF}) that we only use the first four subband coefficients of DHF in the first level.  The corresponding regularization parameters $\lambda_i$ (defined in (\ref{def:Phi1})) are chosen to  adaptively adjust to local variations. Let $\mathcal{I}(i)$ be the set containing all indices  in the neighborhood at the $i$th pixel. Then $\lambda_i$ is set as
\begin{equation} \label{def:lambda}
\lambda_i = \frac{\lambda\times |\mathcal{I}(i)|}{\max\{\sum_{p\in
\mathcal{I}(i)} \|w_p\|,10^{-10}\} } ,
\end{equation}
where $w_p =[v_{(p1)}, v_{(p2)}]^{\top}$ or $w_p =[v_{(p3)}, v_{(p4)}]^{\top}$ are defined as in \eqref{eq:yi-phi1}.  In our tests, we choose the neighborhood of window size $3\times3$ and the parameter $\lambda$ is set by hand.

For the regularization parameters  $\theta_i$ associated with the DCT-based tight framelet coefficients (see (\ref{def:Phi2})), they are all automatically estimated and updated using the approach in our previous work \cite{Li-Shen-Suter:IEEEIP:13}. More precisely, for the regularization parameters $\theta_{i\kappa}$, $\kappa=1,\ldots, 8$, used in \eqref{def:Phi2}, they are automatically estimated according to the local variations of framelet coefficients and noise level.  Suppose the $\epsilon$ in model~\eqref{eq:model} is the Gaussian noise with the standard deviation $\sigma$. As it was done in our previous work \cite{Li-Shen-Suter:IEEEIP:13}, $\sigma_\kappa^2$ the noise variance of the framelet coefficients coming from the filter $\tau_\kappa$ at the second decomposed level is estimated as $\sigma_\kappa^2= \frac{\sigma^2}{4}\|\tau_\kappa\|^2_F$, where $\|\tau_\kappa\|_F$ is the Frobenius norm of $\tau_\kappa$;
$(\sigma_i^\kappa)^2$, the local signal variance of the $i$th framelet coefficients coming from  the filter $\kappa$th, is computed as
$(\sigma_i^\kappa)^2 = \max \{(\sum_{p\in\mathcal{I}^\kappa(i)}|v_p|/|\mathcal{I}^\kappa(i)|)^2-\sigma_\kappa^2,10^{-10}\}$,
where $\mathcal{I}^\kappa(i)$ is the set containing all indices in the neighborhood at the $i$th framelet coefficients from the filter $\kappa$.  With them,
the  regularization parameters are estimated as
\begin{equation}\label{def:theta_i_upd}
    \theta_{i\kappa} = \frac{\sqrt{2}\sigma_{\kappa}^2 }{\sigma_i^\kappa },\quad \kappa=1,\ldots, 8.
\end{equation}

To save computational cost of estimating parameters $\lambda_i$ in \eqref{def:lambda}  and $\theta_{i\kappa}$  in \eqref{def:theta_i_upd}, we only update these parameters when the iteration $k$ is a multiple of $30$ and fix them after the $200$th iteration in our numerical experiments.

\section{Experiments}\label{sec:Experiments}
In this section, we present numerical experiments to illustrate the effectiveness and efficiency of our proposed model \eqref{eq:our-opt}  for image restoration. We use the images ``Square Circle'', ``Cameraman'', and ``Montage'' of size $256 \times 256 $ as the original images $u$ in our experiments, see Figure~\ref{fig:ori_img}.  The pixel values of these images are normalized to the interval $[0,1]$. The quality of the restored image, say $\widetilde{u}$, is evaluated in terms of the peak-signal-to-noise ratio (PSNR) that is defined by
$$
\mathrm{PSNR}:=10 \log_{10} \frac{255^2 n}{\|\widetilde{u}-u\|^2},
$$
where $n$ is the number of pixels in $u$. To incorporate structural information in image comparisons, the metric of structural similarity (SSIM) \cite{Wang-Bovik-Sheikh-Simoncelli:IEEEIP:2004} of $\widetilde{u}$ to $u$ is reported as well. The higher the PSNR and SSIM, the better the quality of the restored image.


\begin{figure}[htbp]
\centering
\begin{tabular}{ccc}
\scalebox{0.5}{\includegraphics*[0,0][255,255]{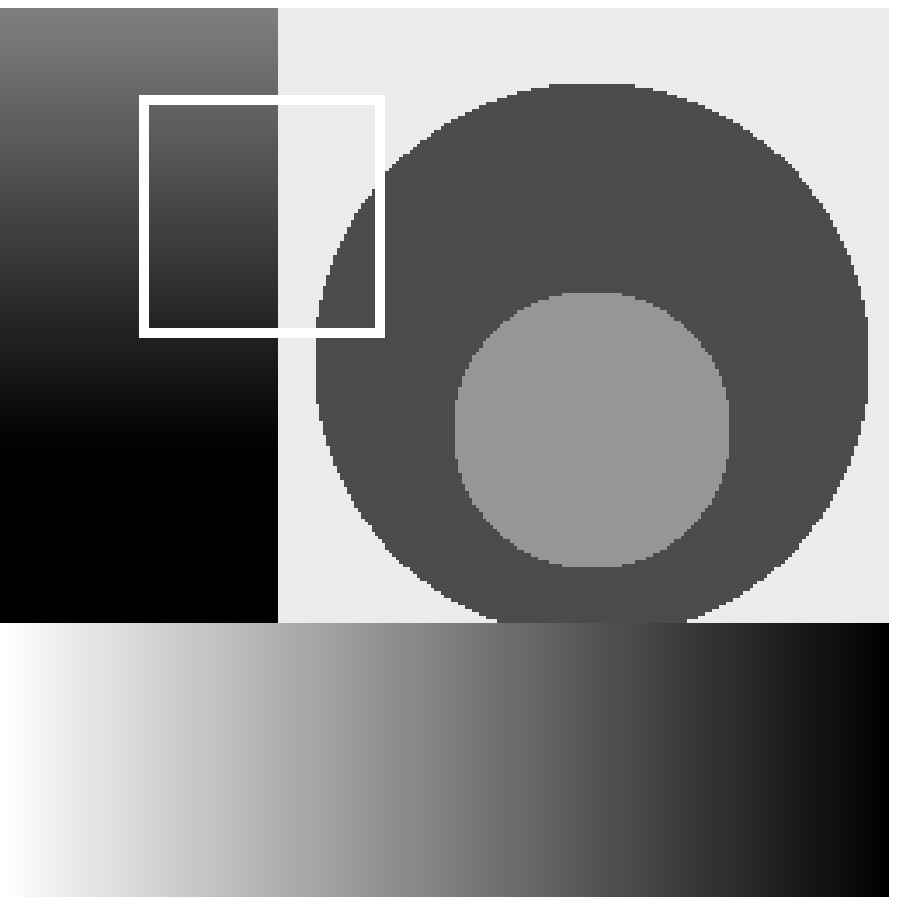}}&
\scalebox{0.5}{\includegraphics*[0,0][255,255]{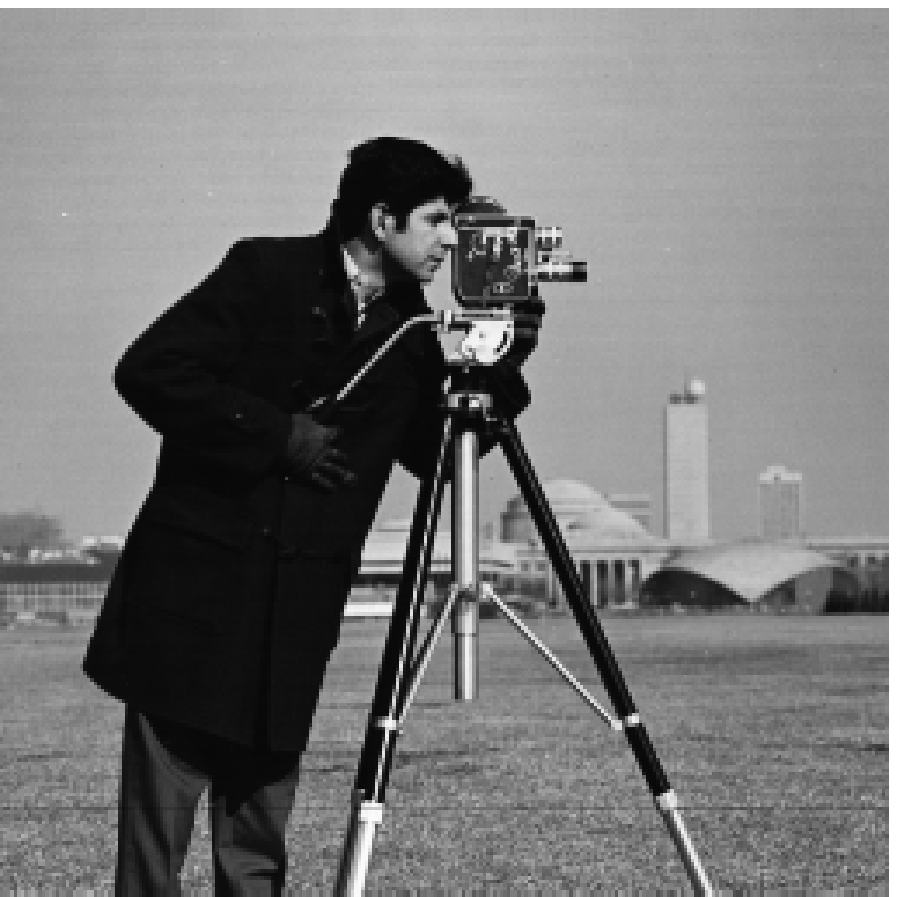}}&
\scalebox{0.5}{\includegraphics*[0,0][255,255]{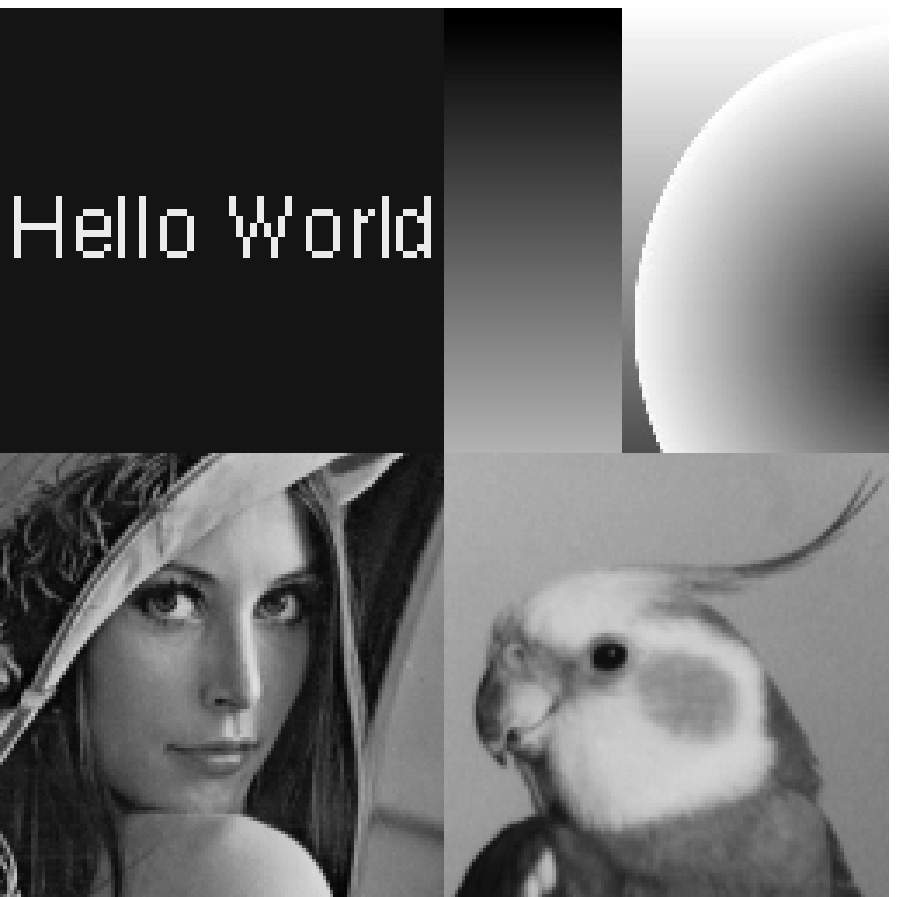}}
\\
(a)&(b)& (c)
\end{tabular}
\caption{Original image: (a) Square Circle; (b) Cameraman; (c) Montage.}
\label{fig:ori_img}
\end{figure}


Two sets of comparisons for image restoration will be conducted in this section. The first set is to compare with other tight frame regularizers. The second set is to compare with some derivative-based models.

\subsection{Comparison with Tight Frame Regularizers}
Here we compare the proposed regularization functional~\eqref{def:OurRegu} with two other tight frame regularization functionals $\mathcal{G}_{\mathrm{DCT}}$ and $\mathcal{G}_{\mathrm{DHF+DCT}}$ while using the classical TV regularizer $\mathcal{G}_{\mathrm{TV}}$ (see (\ref{eq:TV})) as a benchmark. The $\mathcal{G}_{\mathrm{DCT}}$ is defined as $\mathcal{G}_{\mathrm{DCT}}(u)=\Phi_{2\Theta}(B_{2h}u)$ which only uses the DCT-based tight framelet and takes the first- and second-order information on the image $u$, where $\Phi_{2\Theta}$ is given in \eqref{def:Phi2}.
The $\mathcal{G}_{\mathrm{DHF+DCT}}$ is defined as  $\mathcal{G}_{\mathrm{DHF+DCT}}(u)=\Phi_{1\Lambda}(B_{1h}u)+\Phi_{2\Theta}(B_{2h}u)$, where $\Phi_{1\Lambda}$ is given in \eqref{eq:G-1} with $\varphi_1$ in \eqref{eq:varphi-1-DHF}. The main difference between our proposed regularization functional~\eqref{def:OurRegu} and $\mathcal{G}_{\mathrm{DHF+DCT}}$ is that the action $B_{2h}$ takes on the smoothed image $B_{1\ell}u$ for our  regularization functional while the action $B_{2h}$ takes directly on the image $u$ for $\mathcal{G}_{\mathrm{DHF+DCT}}$.

In our experiment, the image of ``Square Circle'' in Figure \ref{fig:ori_img}(a) (which is the same as Figure~\ref{fig:circle_04}(a)) is blurred by a $5 \times 5$ average kernel (using the Matlab command {\sc fspecial('average', [5:5])}, followed by adding Gaussian noise of mean zero and standard deviation $\sigma=0.04$.  The values of the pair of (PSNR, SSIM) of these restored images by $\mathcal{G}_{\mathrm{TV}}$, $\mathcal{G}_{\mathrm{DCT}}$, $\mathcal{G}_{\mathrm{DHF+DCT}}$, and the proposed regularization functional~\eqref{def:OurRegu} are (33.66dB, 0.962), (32.54dB, 0.970), (33.14dB, 0.980), and (35.00dB, 0.980), respectively. To view the visual quality of the restored images, the square portion marked in the image \ref{fig:ori_img}(a) is displayed in  Figure~\ref{fig:circle_04}. For the regions pointed by two arrows, we can conclude that the proposed regularization functional~\eqref{def:OurRegu} leads to the restored images having better visual quality than the others. The results clearly show that our combined tight frame model is better than other intuitive tight frame models.

%
%

\begin{figure}[htbp]
\centering
\begin{tabular}{ccc}
\scalebox{2.20}{\includegraphics*[43,163][107,227]{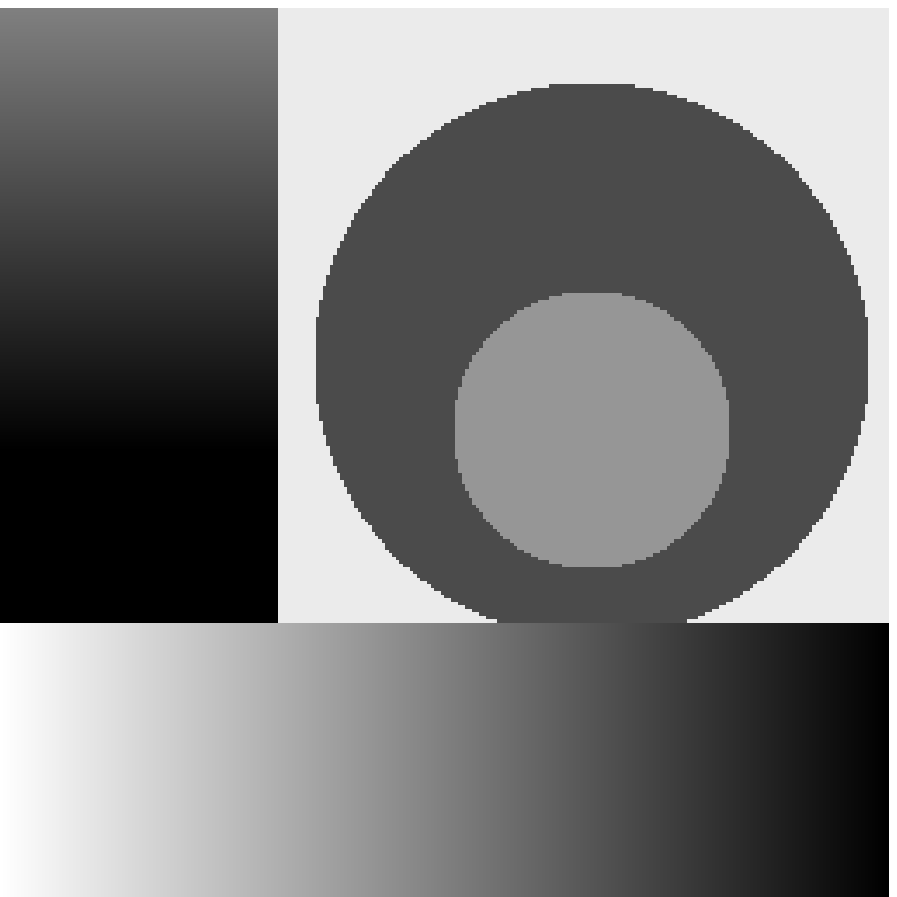}}&
\scalebox{2.20}{\includegraphics*[43,163][107,227]{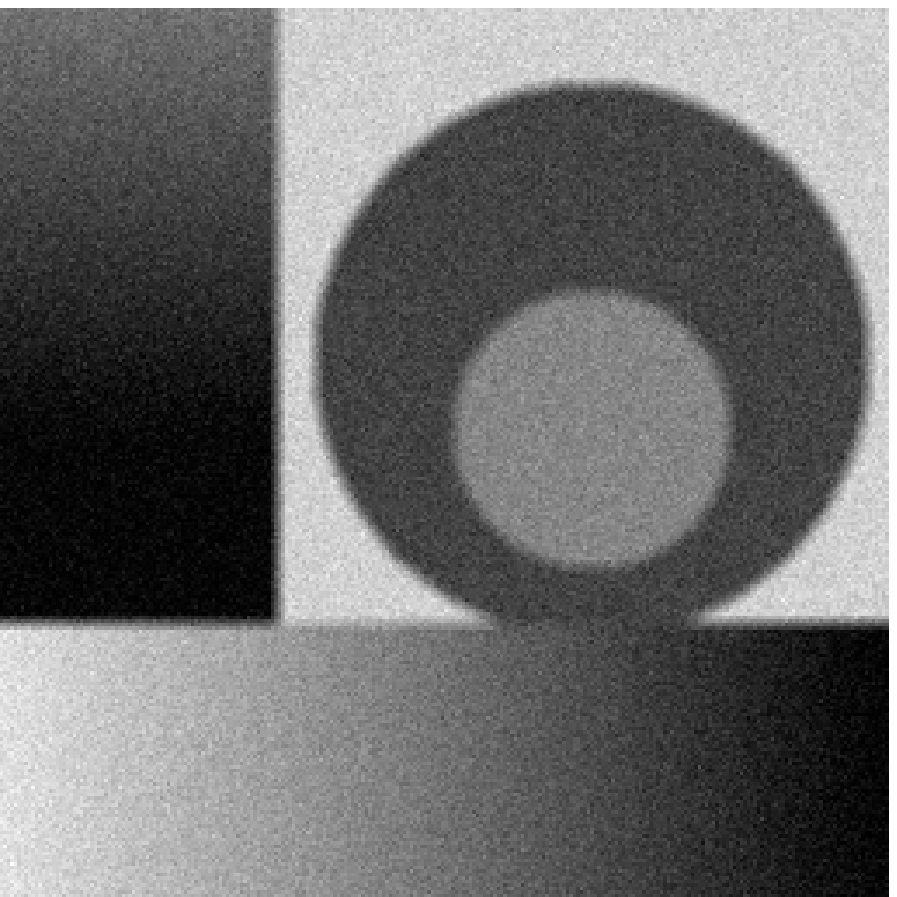}}&
\scalebox{2.20}{\includegraphics*[43,163][107,227]{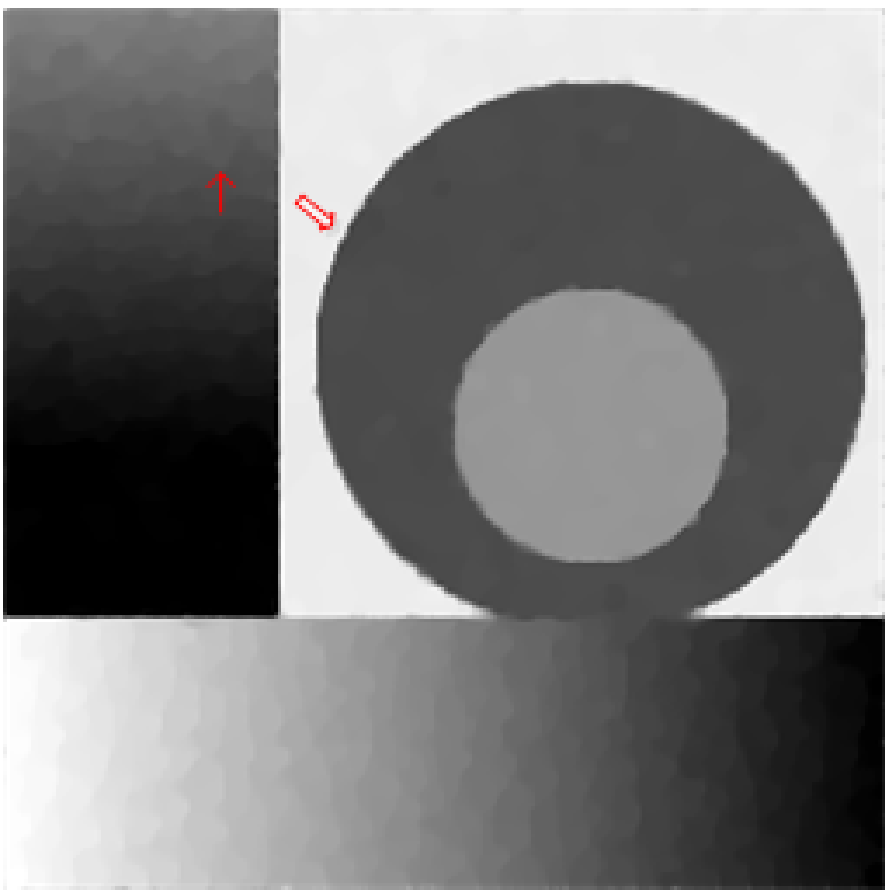}}\\
(a)&(b)&(c)\\
\scalebox{2.20}{\includegraphics*[43,163][107,227]{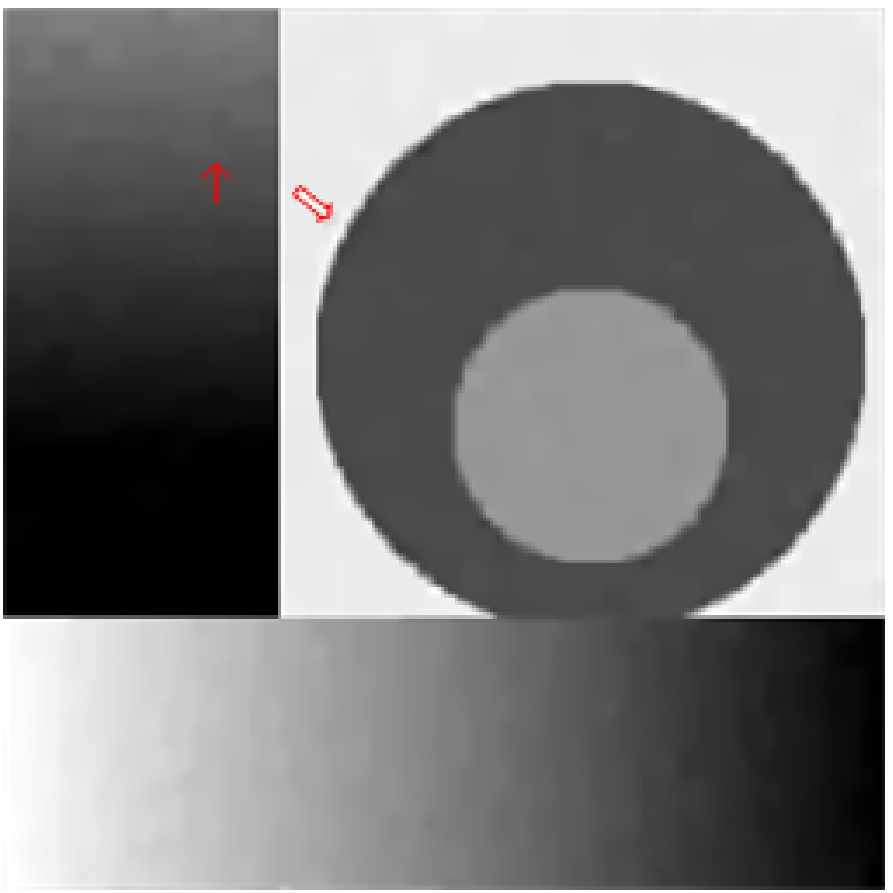}}&
\scalebox{2.20}{\includegraphics*[43,163][107,227]{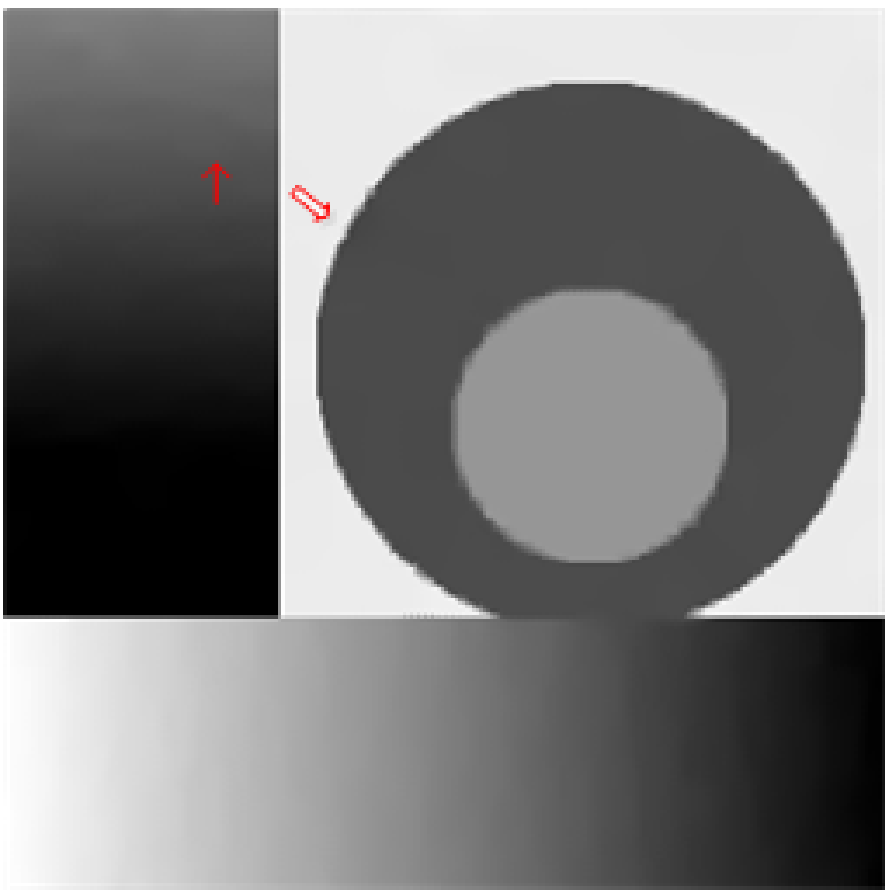}}&
\scalebox{2.20}{\includegraphics*[43,163][107,227]{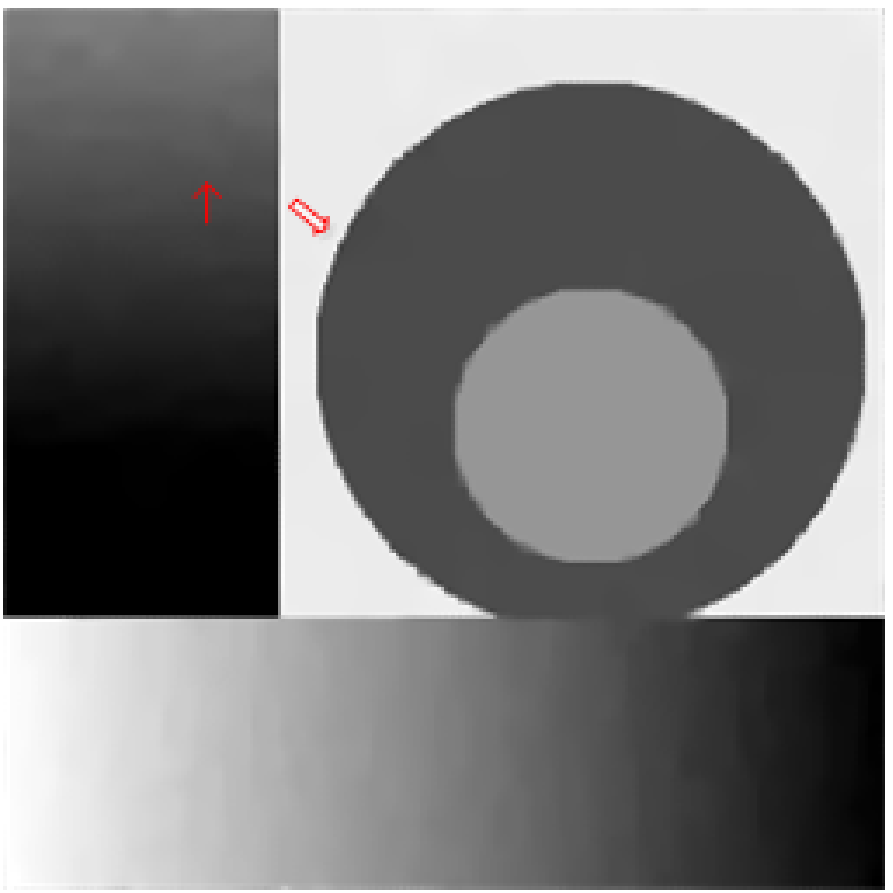}}\\
(d)&(e)&(f)
\end{tabular}
\caption{(a) A region of the image of ``Square Circle''; (b) The blurred image by the kernel {\sc fspecial('average', [5:5])} with Gaussian noise of mean zero and variance $\sigma=0.04$;  The restored images with regularization (c) TV with $\alpha = 0.04$
(see (\ref{eq:finalmodel}) and (\ref{eq:TV}));
(d) $\mathcal{G}_{\mathrm{DCT}}$; (e) $\mathcal{G}_{\mathrm{DHF+DCT}}$; and (f) the proposed regularization functional~\eqref{def:OurRegu} with $\lambda = 0.00035$ (see \eqref{def:lambda}) in the first level, respectively.}
\label{fig:circle_04}
\end{figure}

\subsection{Comparison with Derivative-based Regularizers}

Now we give a comprehensive comparison between our model \eqref{eq:our-opt} and the TV and TGV models. The TV model uses $\mathcal{G}(u)$ in \eqref{eq:TV}
as its regularization term while the TGV  model uses $\mathcal{G}(u)$ in \eqref{eq:TGV-regu} as its regularization term. The software of  TGV  model was provided by the authors in \cite{Guo-Qin-Yin:SIAMIS:14}. All algorithms are carried out until the stopping condition $\|u^{(k+1)}-u^{(k)}\|_{2}/\| u^{(k)}\|_{2}<10^{-9}$ is satisfied or the maximal iterations is $400$.

In our experiments, the test images in Figure \ref{fig:ori_img} are blurred by a $5 \times 5$ average kernel (using the Matlab command
{\sc fspecial('average', [5:5])}, followed by adding Gaussian noise of mean zero and standard deviation $\sigma$. For different values of $\sigma$, the  PSNR and SSIM values  of the restored images by TV, TGV, and our TNTF  are  reported in Table \ref{Table:GauNoise_Indexes}. The highest values of PSNR and SSIM for each $\sigma$ in each test image are highlighted. It clearly shows that our proposed TNTF performs the best in terms of both PSNR and SSIM values. We remark that the regularization parameters in the second level for our TNTF are automatically estimated based on the approach in our work \cite{Li-Shen-Suter:IEEEIP:13}.

\begin{table*}[htbp]\small
\centering \caption{The PSNR ($dB$) and SSIM  for the restored results of each algorithm with blurred
images contaminated by Gaussian noise. The test images
are blurred by the blurring kernel {\sc {fspecial('average',[5:5])}}.
}\label{Table:GauNoise_Indexes}
\begin{tabular}{c|cc|cc|cc|c}
  \hline
  \multirow{1}{*}{Algorithm} &  \multicolumn{2}{c|}{`` Square Circle''}&\multicolumn{2}{c|}{``Cameraman''}
  &\multicolumn{2}{c|}{``Montage''}& Case\\ \cline{2-7}
  & PSNR & SSIM & PSNR & SSIM & PSNR & SSIM &   \\ \cline{1-8}
TV & 35.40dB & 0.980& 26.43dB & 0.815& 28.21dB & 0.907& \\
TGV &35.58dB & 0.976& 26.27dB &0.811& 28.84dB & 0.910& STD $\sigma$=0.02  \\
TNTF & \textbf{38.19}dB & \textbf{0.992}& \textbf{27.06}dB & \textbf{0.821}& \textbf{30.19}dB & \textbf{0.92}4&     \\
  \hline
TV & 34.51dB & 0.969& 25.64dB & 0.791& 26.92dB & 0.886& \\
TGV  &34.33dB & 0.967& 25.58dB &0.788& 26.93dB & 0.884& STD $\sigma$=0.03  \\
TNTF & \textbf{36.14}dB & \textbf{0.985}& \textbf{26.01}dB & \textbf{0.800}& \textbf{28.91}dB & \textbf{0.910}&     \\
  \hline
TV & 33.66dB & 0.962& 25.10dB & 0.774& 25.94dB & 0.867& \\
TGV &33.41dB & 0.955& 24.94dB &0.770& 26.17dB & 0.875& STD $\sigma$=0.04  \\
TNTF & \textbf{35.00}dB & \textbf{0.980}& \textbf{25.31}dB & \textbf{0.784}& \textbf{27.88}dB & \textbf{0.898}&     \\
  \hline
\end{tabular}
\end{table*}

In the rest of this section, we provide qualitative results of the restored images from the above three algorithms. We first show the case for the blurred image of ``Square Circle" with Gaussian noise of STD $\sigma=0.03$ in Figure~\ref{fig:circle_03}. The noisy and blurry image is shown in Figure~\ref{fig:circle_03}(a).
The regularization parameter $\alpha$ = 0.02 (see (\ref{eq:finalmodel}) and (\ref{eq:TV})) is used for the TV model and $(\alpha_1,\alpha_2) = (0.0105, 0.026)$ (see (\ref{eq:TGV-regu})) is used for the TGV model based on the best achievable PSNR values.  The regularization parameter $\lambda=0.0002$ in the first level (see (\ref{def:lambda})) is used in our proposed model. We can observe that Figure~\ref{fig:circle_03}(b) produced by the TV has lots of staircase artifacts even without zooming in. As we can see from Figure~\ref{fig:circle_03}(c) and (d), this kind of staircase artifacts is significantly reduced by the TGV and TNTF. To have a closer look at the visual quality of the restored images by various algorithms, two parts of Figure \ref{fig:circle_03} are zoomed in and displayed in the first column of  Figure~\ref{fig:circle_03_Zoomin4}. The corresponding parts in the restored images by TV, TGV, and TNTF are shown in Figure~\ref{fig:circle_03_Zoomin4}(b), (c), and (d), respectively. We can conclude that the horizontal line in the image is well preserved by the TNTF.
\begin{figure}[htbp]
\centering
\begin{tabular}{cc}
\scalebox{0.5}{\includegraphics*[0,0][255,255]{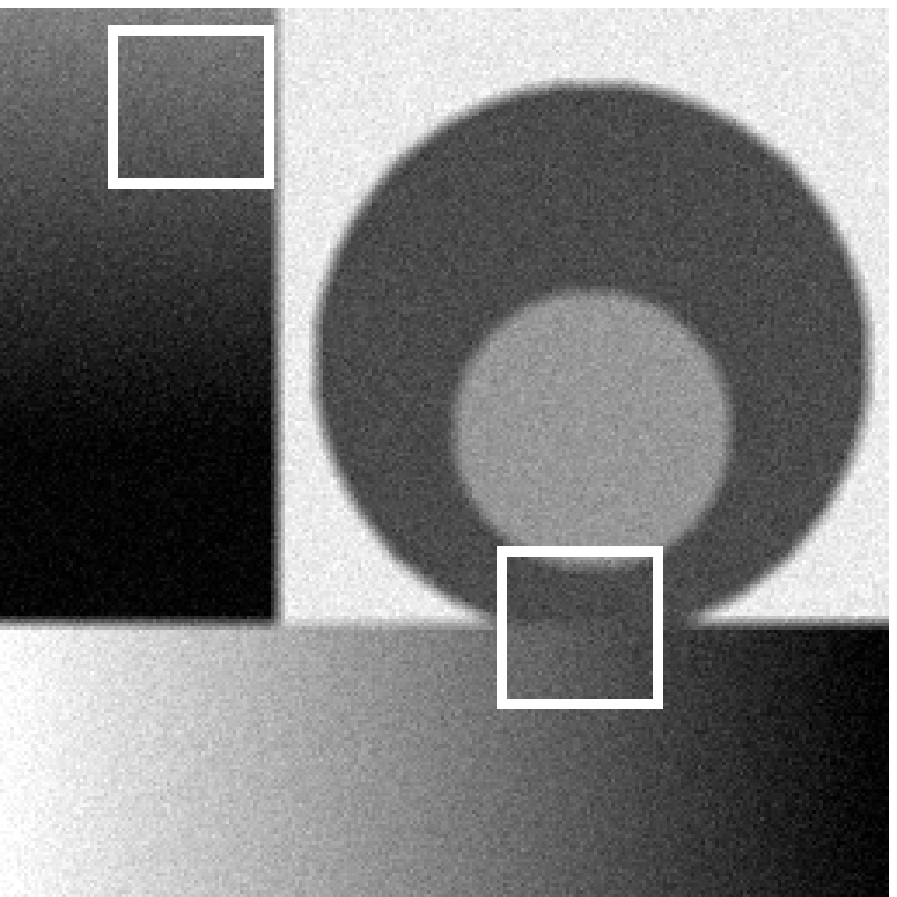}}&
\scalebox{0.5}{\includegraphics*[0,0][255,255]{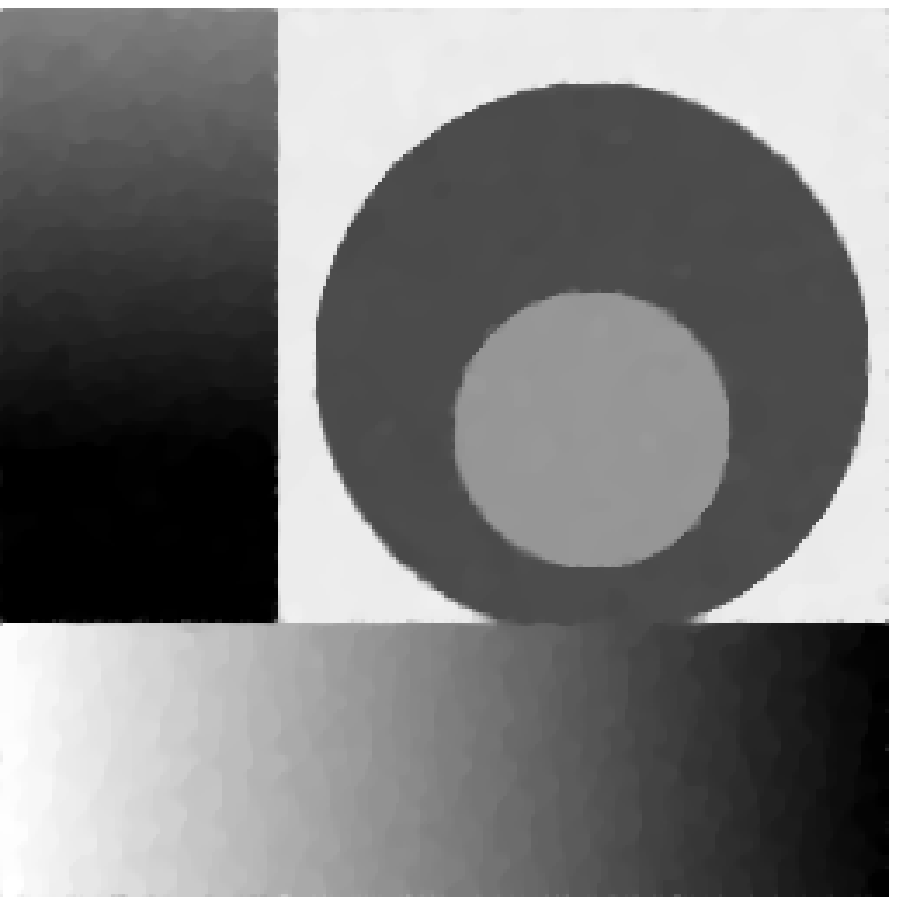}}\\
(a)&(b)\\
\scalebox{0.5}{\includegraphics*[0,0][255,255]{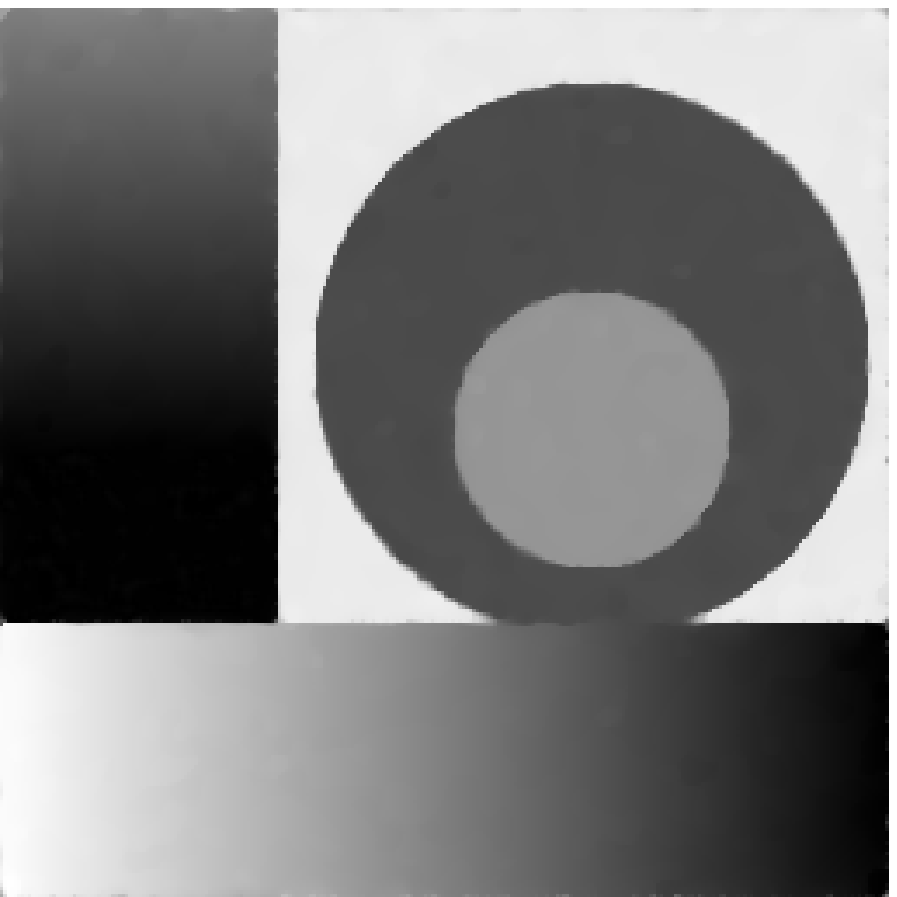}}&
\scalebox{0.5}{\includegraphics*[0,0][255,255]{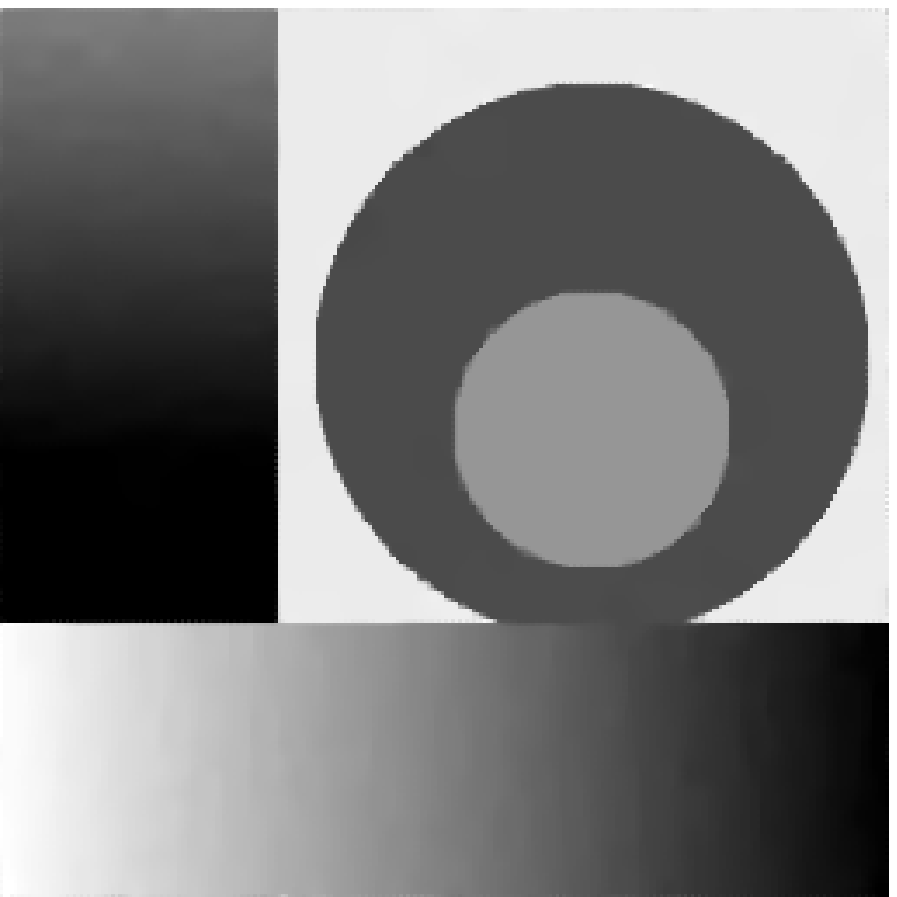}}\\
(c)&(d)
\end{tabular}
\caption{(a) Image blurred by kernel {\sc fspecial('average', [5:5])} and added Gaussian noise with $\sigma=0.03$;  Images reconstructed by (b) TV with $\alpha = 0.02$ (see (\ref{eq:finalmodel}) and (\ref{eq:TV})),
(c) TGV with $(\alpha_1,\alpha_2) = (0.0105,0.026)$ (see (\ref{eq:TGV-regu})),  and (d) TNTF with $\lambda = 0.0002$ (see (\ref{def:lambda})), respectively.}
\label{fig:circle_03}
\end{figure}

\begin{figure}[htbp]
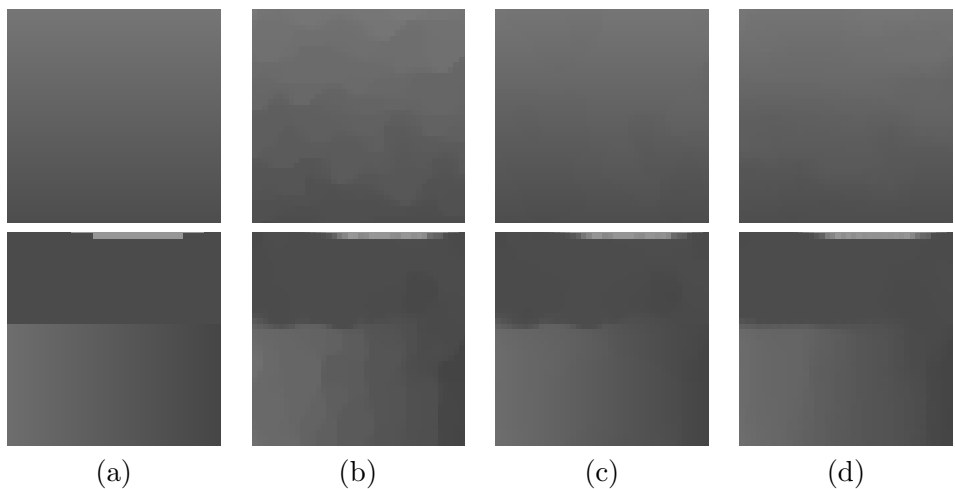

\centering
\begin{tabular}{cccc}
\scalebox{2.00}{\includegraphics*[34,206][74,246]{./Fig4.eps}}&
\scalebox{2.00}{\includegraphics*[34,206][74,246]{./Fig11.eps}}&
\scalebox{2.00}{\includegraphics*[34,206][74,246]{./Fig12.eps}}&
\scalebox{2.00}{\includegraphics*[34,206][74,246]{./Fig13.eps}}\\
\scalebox{2.00}{\includegraphics*[146,56][186,96]{./Fig4.eps}}&
\scalebox{2.00}{\includegraphics*[146,56][186,96]{./Fig11.eps}}&
\scalebox{2.00}{\includegraphics*[146,56][186,96]{./Fig12.eps}}&
\scalebox{2.00}{\includegraphics*[146,56][186,96]{./Fig13.eps}}\\
(a) & (b) & (c) & (d)
\end{tabular}
\caption{Two zoom-in parts of Fig.~\ref{fig:circle_03}: (a) Original image;  images reconstructed by (b) TV,
(c) TGV,  and (d) TNTF, respectively.}
\label{fig:circle_03_Zoomin4}
\end{figure}

Figure~\ref{fig:Cameraman_02}(a) is the blurred image of ``Cameraman'' corrupted by Gaussian noise of STD $\sigma$=0.02. The restored images by TV, TGV, and TNTF are displayed in Figure~\ref{fig:Cameraman_02}(b), (c), and (d), respectively. The regularization parameters for TV, TGV, and TNGV are $\alpha = 0.006$,  $(\alpha_1,\alpha_2) = (0.0035,0.0095)$, and $\lambda= 0.0004$, respectively. The structures of the building as well as the man are well preserved in the restored image by our TNTF. Block artifacts are clearly observed in the sky of the restored images by TV and TGV (see  Figure~\ref{fig:Cameraman_02}(b), (c)), but not in Figure~\ref{fig:Cameraman_02}(d). The zoom-in part of Figure~\ref{fig:Cameraman_02} is displayed in Figure~\ref{fig:Cameraman_zoomin}. The shape of the camera lens in the restored image by TNTF is more closer to the original one than that by TV and TGV.

\begin{figure}[htbp]
\centering
\begin{tabular}{cc}
\scalebox{0.50}{\includegraphics*[0,0][255,255]{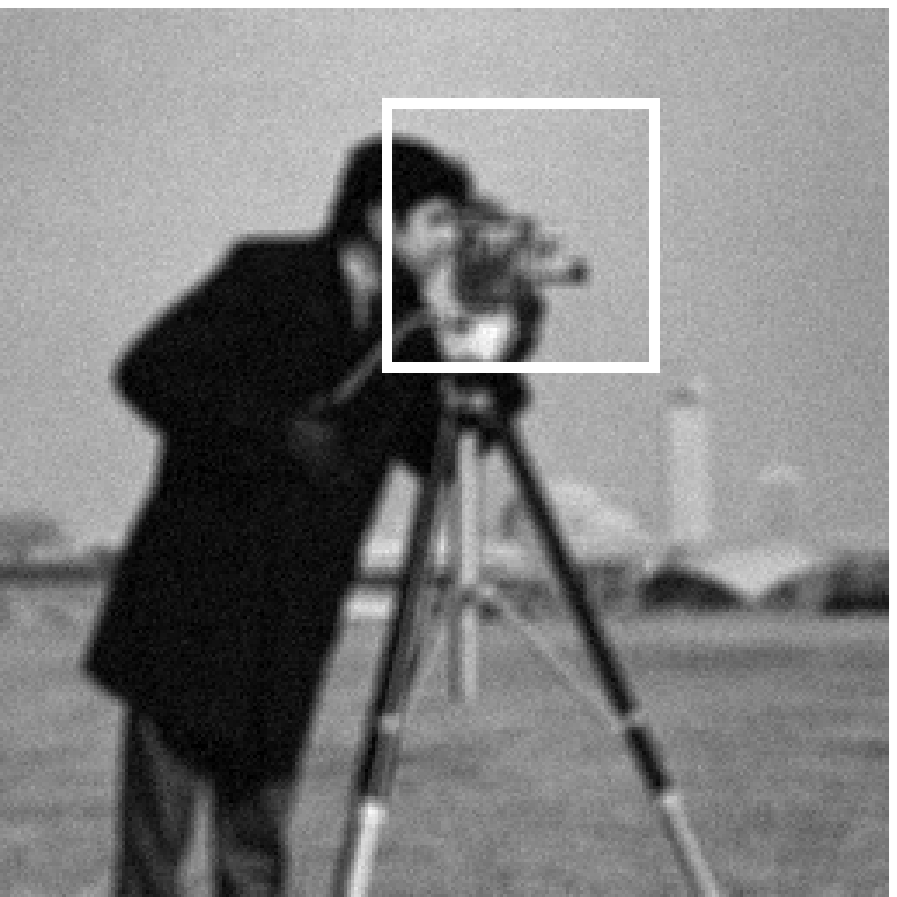}}&
\scalebox{0.50}{\includegraphics*[0,0][255,255]{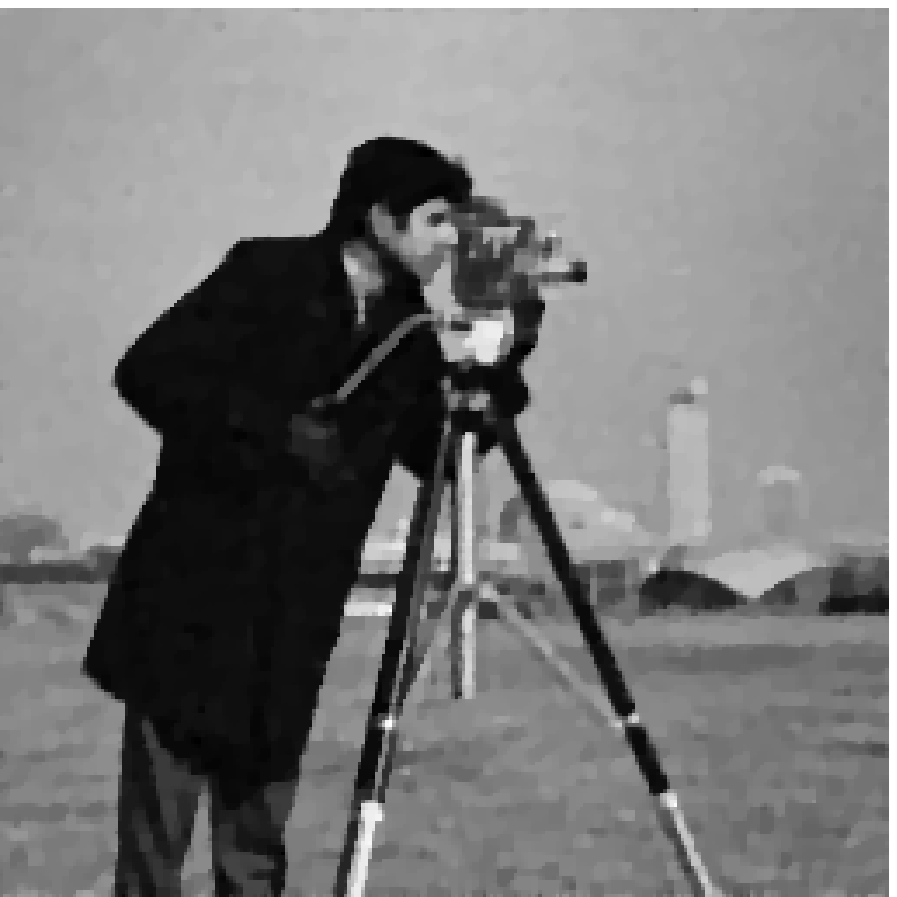}}\\
(a)&(b)\\
\scalebox{0.50}{\includegraphics*[0,0][255,255]{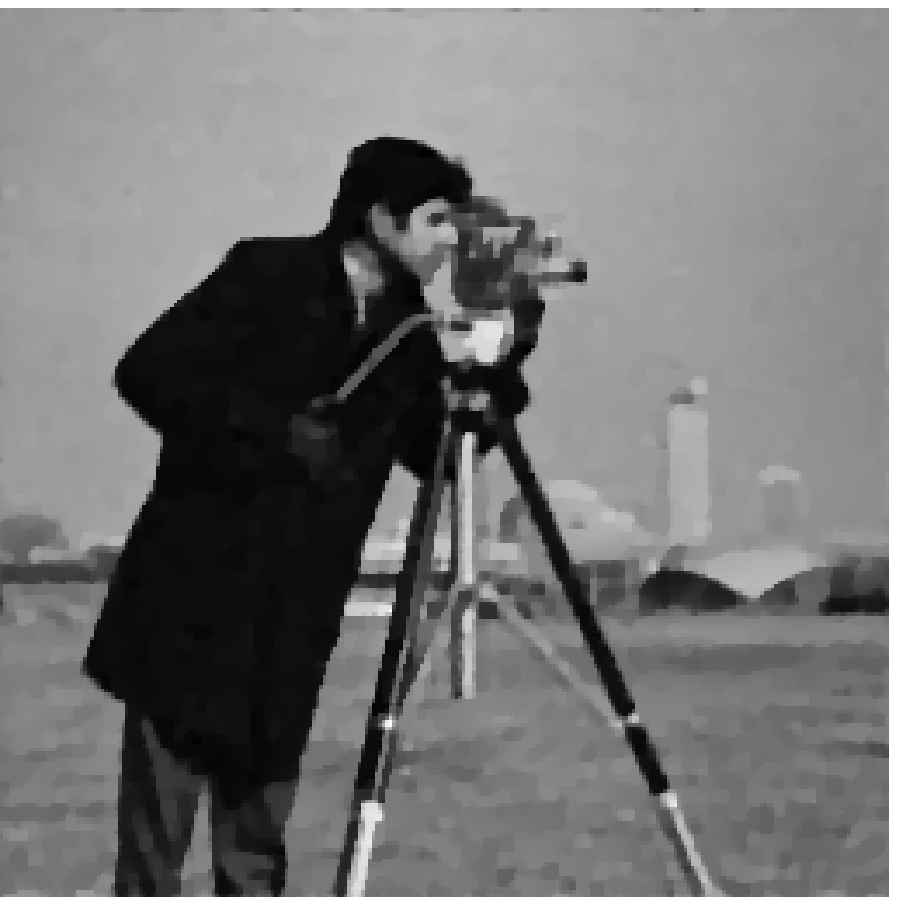}}&
\scalebox{0.50}{\includegraphics*[0,0][255,255]{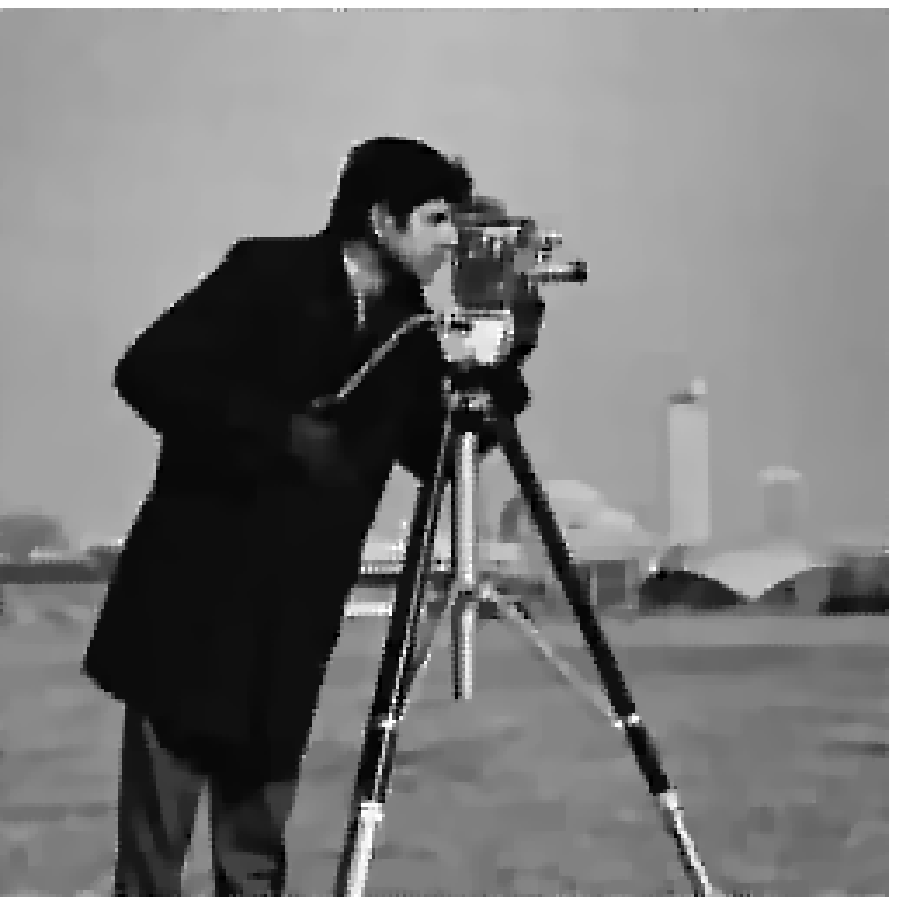}}\\
(c)&(d)
\end{tabular}
\caption{(a) Image blurred by kernel {\sc fspecial('average', [5:5])} and added Gaussian noise with $\sigma=0.02$; images reconstructed by (b) TV with $\alpha$ = 0.006;
(c) TGV with $(\alpha_1,\alpha_2) = (0.0035,0.0095)$,  and (d) TNTF with $\lambda = 0.0004$, respectively.}
\label{fig:Cameraman_02}
\end{figure}

\begin{figure}[htbp]
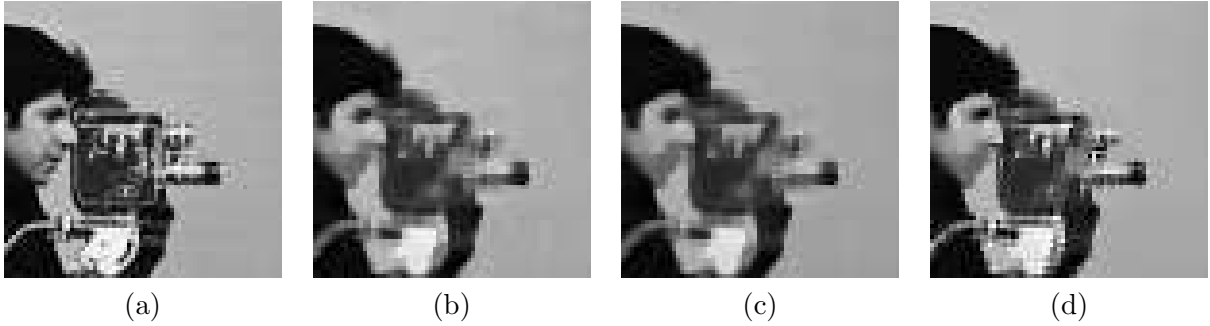

\centering
\begin{tabular}{cccc}
\scalebox{1.45}{\includegraphics*[113,153][185,225]{./Fig2.eps}}&
\scalebox{1.45}{\includegraphics*[113,153][185,225]{./Fig15.eps}}&
\scalebox{1.45}{\includegraphics*[113,153][185,225]{./Fig16.eps}}&
\scalebox{1.45}{\includegraphics*[113,153][185,225]{./Fig17.eps}}\\
(a)&(b)&(c)&(d)
\end{tabular}
\caption{Zoom-in parts of Figure~\ref{fig:Cameraman_02}: (a) Original image; images reconstructed by (b) TV,  (c) TGV,   and (d) TNTF.}
\label{fig:Cameraman_zoomin}
\end{figure}

Figure~\ref{fig:montage_04}(a) is the blurred image of ``Montage'' corrupted by Gaussian noise of STD $\sigma$=0.04. The restored images by TV, TGV, and TNTF are displayed in Figure~\ref{fig:montage_04}(b), (c), and (d), respectively. The regularization parameters for TV, TGV, and TNGV are $\alpha = 0.019$,  $(\alpha_1,\alpha_2) = (0.0085,0.0155)$, and $\lambda= 0.00015$, respectively. Severe artifacts appeared in Figure~\ref{fig:montage_04}(b) by TV, are significantly suppressed in Figure~\ref{fig:montage_04}(c) and (d) by TGV and TNTF. Two zoom-in parts of the results are shown in Figure~\ref{fig:montage_zoomin}. It is evident that the lines are well preserved in Figure~\ref{fig:montage_zoomin}(d) by TNTF.

\begin{figure}[htbp]
\centering
\begin{tabular}{cc}
\scalebox{0.50}{\includegraphics*[0,0][255,255]{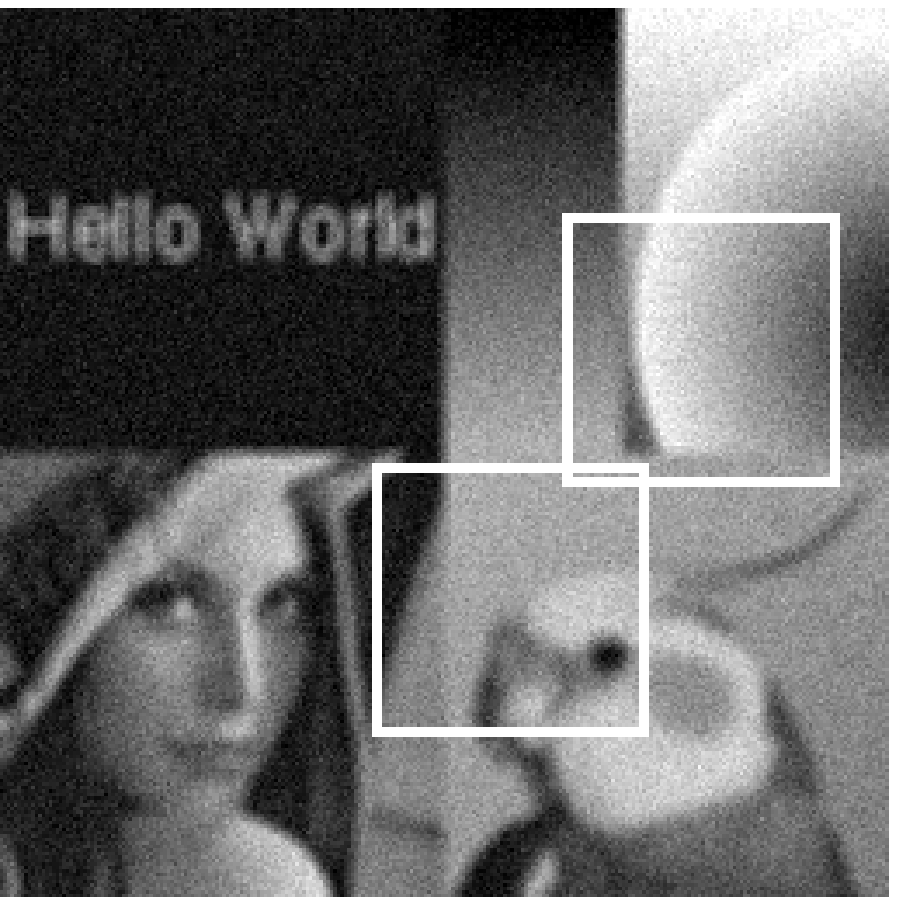}}&
\scalebox{0.50}{\includegraphics*[0,0][255,255]{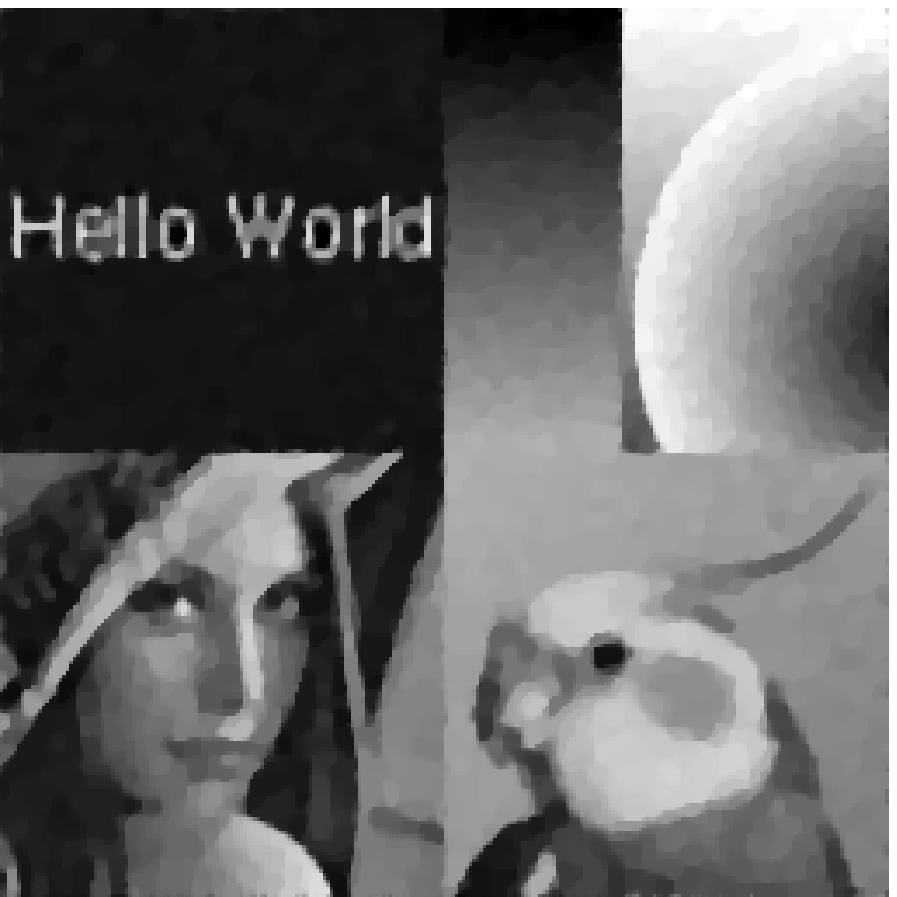}}\\
(a)&(b)\\
\scalebox{0.50}{\includegraphics*[0,0][255,255]{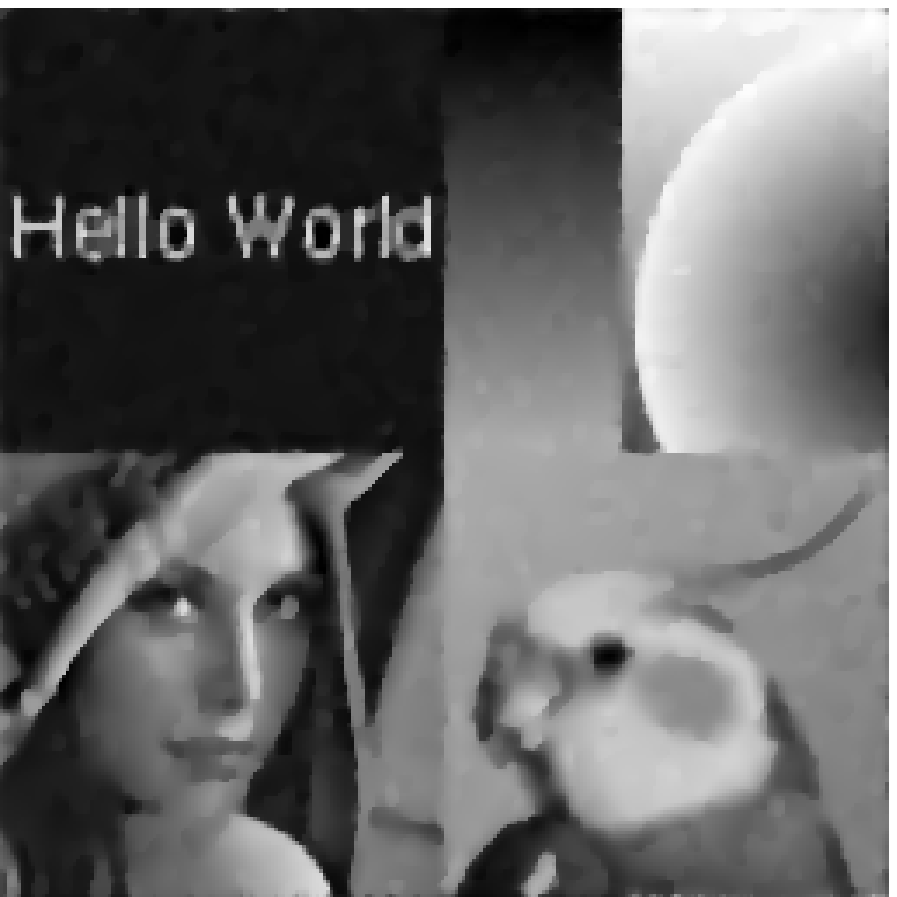}}&
\scalebox{0.50}{\includegraphics*[0,0][255,255]{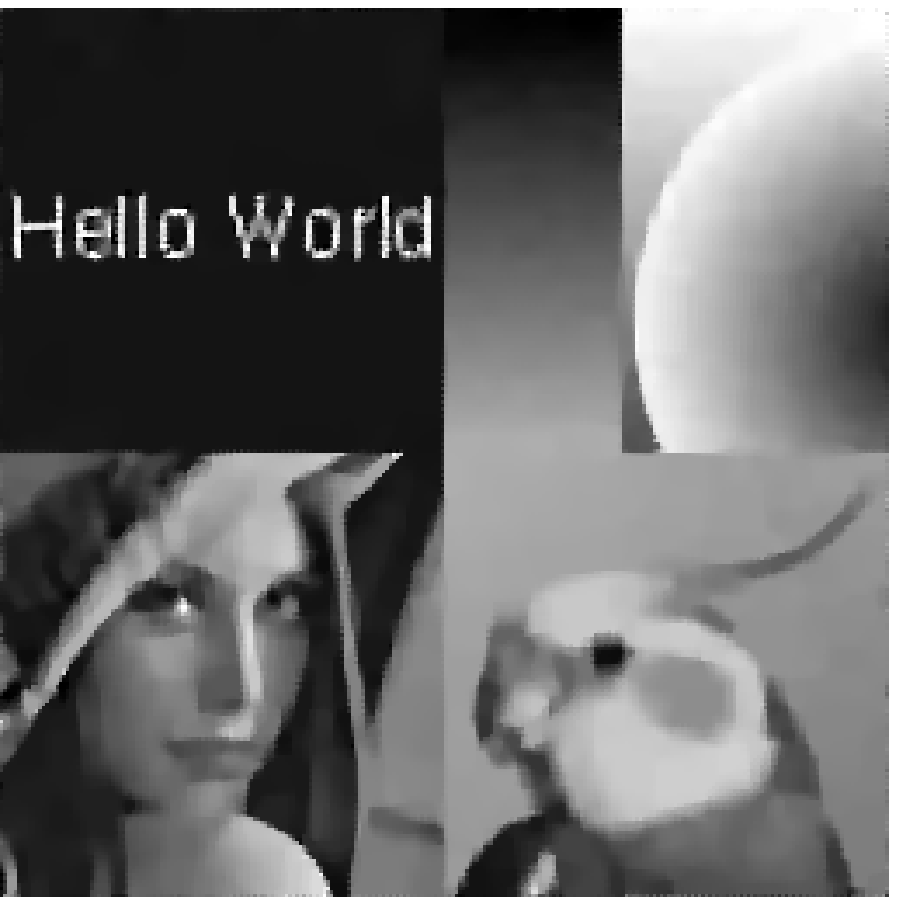}}\\
(c)&(d)
\end{tabular}
\caption{(a) Image blurred by kernel {\sc fspecial('average', [5:5])} and added Gaussian noise with $\sigma=0.04$; images reconstructed by (b) TV with $\alpha = 0.019$,
(c) TGV with $(\alpha_1,\alpha_2) = (0.0035,0.0095)$,  and (d) TNTF with $\lambda = 0.00015$.}
\label{fig:montage_04}
\end{figure}

\begin{figure}[htbp]
\centering
\begin{tabular}{cccc}
\scalebox{1.45}{\includegraphics*[110,48][182,120]{./Fig3.eps}}&
\scalebox{1.45}{\includegraphics*[110,48][182,120]{./Fig19.eps}}&
\scalebox{1.45}{\includegraphics*[110,48][182,120]{./Fig20.eps}}&
\scalebox{1.45}{\includegraphics*[110,48][182,120]{./Fig21.eps}}\\
\scalebox{1.45}{\includegraphics*[165,120][237,192]{./Fig3.eps}}&
\scalebox{1.45}{\includegraphics*[165,120][237,192]{./Fig19.eps}}&
\scalebox{1.45}{\includegraphics*[165,120][237,192]{./Fig20.eps}}&
\scalebox{1.45}{\includegraphics*[165,120][237,192]{./Fig21.eps}}\\
(a)&(b)& (c)&(d)
\end{tabular}
\caption{Two zoom-in parts of Fig.~\ref{fig:montage_04}: (a) Original image;  images reconstructed by (b) TV, (c) TGV,  and (d) TNTF.}
\label{fig:montage_zoomin}
\end{figure}

\section{Conclusion}
In this paper, we have designed a two-level non-stationary tight framelet system and utilized it in a regularization model for image restoration. This framelet system has the ability to capture the first and second order information of the image to be reconstructed. We developed an algorithm to solve the resulting optimization problem. The numerical experiments show the effectiveness of the proposed image restoration model and the corresponding algorithm.


\end{document}